\documentclass[hidelinks,onefignum,onetabnum]{siamart220329}

\usepackage{lipsum}
\usepackage{amsfonts}
\usepackage{graphicx}
\usepackage{epstopdf}
\ifpdf

\usepackage{algpseudocode}  
\algnewcommand\Input{\\ \textbf{Input: }}

\newsiamremark{conjecture}{Conjecture}

\newcommand{\argmin}{\operatornamewithlimits{\mathrm{arg\,min}}}
\newcommand{\Exp}{\operatorname{\mathrm{Exp}}}

\DeclareGraphicsExtensions{.eps,.pdf,.png,.jpg}
\else
  \DeclareGraphicsExtensions{.eps}
\fi

\newsiamremark{remark}{Remark}
\newsiamremark{hypothesis}{Hypothesis}
\crefname{hypothesis}{Hypothesis}{Hypotheses}
\newsiamthm{claim}{Claim}

\headers{The injectivity radius of the Stiefel manifold}{P.-A. Absil and Simon Mataigne}

\title{The ultimate upper bound on the injectivity radius of the Stiefel manifold\thanks{Submitted to the editors March 7, 2024.
\funding{This work was supported by the Fonds de la Recherche Scientifique-FNRS under Grant no T.0001.23. Simon Mataigne is a Research Fellow of the FNRS.}}}

\author{P.-A. Absil\thanks{ICTEAM Institute, UCLouvain, B-1348 Louvain-la-Neuve, Belgium 
  (\email{pa.absil@uclouvain.be}, \email{simon.mataigne@uclouvain.be}, \url{http://sites.uclouvain.be/absil/}).}
\and Simon Mataigne\footnotemark[2]}

\usepackage{amsopn}

\ifpdf
\hypersetup{
  pdftitle={The ultimate upper bound on the injectivity radius of the Stiefel manifold},
  pdfauthor={P.-A. Absil and Simon Mataigne}
}
\fi

\begin{document}

\maketitle

\begin{abstract}
    We exhibit conjugate points on the Stiefel manifold endowed with any member of the family of Riemannian metrics introduced by H\"uper et al.~(2021). This family contains the well-known canonical and Euclidean metrics. An upper bound on the injectivity radius of the Stiefel manifold in the considered metric is then obtained as the minimum between the length of the geodesic along which the points are conjugate and the length of certain geodesic loops. Numerical experiments support the conjecture that the obtained upper bound is in fact equal to the injectivity radius.
\end{abstract}

\begin{keywords}
    Stiefel manifold, canonical metric, Euclidean metric, conjugate points, geodesic loops, injectivity radius
\end{keywords}

\begin{MSCcodes}
    15B10, 
    15B57,
    65F99,
    53C22,
    53C30
\end{MSCcodes}

\section{Introduction}
\label{sec:introduction}

The \emph{injectivity radius} of a complete Riemannian manifold is the largest $r$ such that every geodesic of length less than or equal to $r$ is a shortest curve between its endpoints. Equivalent definitions can be found in Riemannian geometry textbooks such as~\cite{dC92,Sak96}.

The knowledge of the injectivity radius is an asset in several computational problems on manifolds. If, between two given points, one finds a geodesic whose length is less than or equal to the injectivity radius, then the length of the geodesic gives the Riemannian distance between the two points. The injectivity radius also defines the largest ball on which geodesic normal coordinates may be used. 
The injectivity radius is involved, e.g., in optimization~\cite{ABG2007-FoCM,AfsariTronVidal2013,Bonnabel2013,HosseiniUschmajew2017,SunFlammarionFazel2019}, consensus~\cite{TronAfsariVidal2013}, statistics~\cite{Afsari2011,ArnaudonNielsen2013,GuiguiMiolanePennec2023}, and data fitting~\cite{AbsGouStrWir2016,GouMasAbs2018}.

The \emph{Stiefel manifold} of orthonormal $p$-frames in $\mathbb{R}^n$, denoted by $\mathrm{St}(n,p)$, is one of the most classical matrix manifolds~\cite{EAS98,AMS2008,boumal2023intromanifolds}. It appears in numerous applications; see for example~\cite{TurVeeSri2011,WenYin2012,BoulangerSaidBihanManton2016,ChakrabortyVemuri2019,ChenMaSoZhang2020,LiJiaWenLiuTao2021,sutti2023shooting}. In particular, the endpoint geodesic problem on the Stiefel manifold has attracted a growing interest; see~\cite{ZimmermannHueper2022,sutti2023shooting} and references therein. 

It may thus seem paradoxical that the injectivity radius of the Stiefel manifold remained unknown at the time of submitting this paper. This has to do with the fact that the cut locus of Riemannian manifolds can have an intricate structure~\cite{ItohTanaka1998}. 

In this paper, we report the discovery of conjugate points on the Stiefel manifold $\mathrm{St}(n,p)$, with $2 \leq p \leq n-2$, for all members of the one-parameter ($\beta>0$) family of Riemannian metrics~\eqref{eq:g} introduced by H\"uper et al.~\cite{HuperMarkinaLeite2021}, which includes the well-known canonical ($\beta=\frac12$) and Euclidean ($\beta=1$) metrics defined by Edelman et al.~\cite{EAS98}. 
The length of the geodesic along which the points are conjugate gives an upper bound on the injectivity radius, specifically $t^{\mathrm{r}}_\beta\sqrt2$ where $t^{\mathrm{r}}_\beta$ is the smallest positive root of $\frac{\sin t}{t} + \frac{1-\beta}{\beta} \cos t$. For a range of values of $\beta$ (see~\eqref{eq:inj-beta}), this bound is further improved by considering the length of geodesic loops generated by Givens rotations. The resulting $\beta$-dependent upper bound on the injectivity radius is given in Theorem~\ref{thm:inj}.
In particular, for the canonical metric, the upper bound is $t^{\mathrm{r}}\sqrt2 \approx 0.91326189159122 \, \pi$, where $t^{\mathrm{r}}$ is the smallest positive root of $\frac{\sin t}{t} + \cos t$. For the Euclidean metric, the upper bound is $\pi$.   

As for the word ``ultimate'' in the title of this paper, it is meant to convey that, according to numerical experiments, the obtained upper bound appears to be \emph{equal} to the injectivity radius for all $\beta$ (Conjecture~\ref{conj:inj}). 
At the time of submitting this paper, we were not aware of a proof of the conjecture for any value of $\beta$. Shortly thereafter, it was proved in the eprint~\cite{zimmermann2024injectivity} that the injectivity radius of the Stiefel manifold under the Euclidean metric is $\pi$, thereby proving the conjecture for $\beta = 1$.

The paper is organized as follows. After background information on the injectivity radius of Riemannian manifolds (section~\ref{sec:inj-Riemannian}) and the Stiefel manifold (section~\ref{sec:Stiefel}), section~\ref{sec:lemmas} proves linear algebra lemmas that lead in section~\ref{sec:conjugate} to an upper bound on the conjugate radius. Together with upper bounds on the length of geodesic loops (section~\ref{sec:loops}), this produces an upper bound on the injectivity radius (section~\ref{sec:inj-upper}). The conjecture that the upper bound is equal to the injectivity radius is supported by numerical experiments in section~\ref{sec:numerical}.

\section{The injectivity radius of a complete Riemannian manifold}
\label{sec:inj-Riemannian}

This section gives a concise overview, based on~\cite{Sak96,CE75}, of the concept of injectivity radius. The necessary background in Riemannian geometry (e.g., tangent spaces, induced norm, induced distance, geodesics, and minimal geodesics, also called minimizing geodesics) can be found in introductory textbooks such as~\cite{dC92}.

Let $\mathcal{M}$ be a complete Riemannian manifold with Riemannian metric $g$ and induced distance $d$. The \emph{exponential map} at $x$ is 
\[
    \mathrm{Exp}_x: \mathrm{T}_x\mathcal{M} \to \mathcal{M} : \xi \mapsto \gamma_\xi(1),
\]
where $\mathrm{T}_x\mathcal{M}$ denotes the tangent space to $\mathcal{M}$ at $x$ and $\gamma_\xi:[0,\infty) \to \mathcal{M}$ is the geodesic of $\mathcal{M}$ satisfying  $\gamma(0)=x$ and $\gamma'(0)=\xi$. It is customary to write $\Exp(\xi)$ instead of $\Exp_x(\xi)$ when it is clear that $\xi$ belongs to $\mathrm{T}_x\mathcal{M}$. 

Let $\mathrm{U}_x\mathcal{M} = \{\xi \in \mathrm{T}_x\mathcal{M} \mid g_x(\xi,\xi) = 1\}$ denote the unit tangent space at $x\in\mathcal{M}$. Given $\xi \in \mathrm{U}_x\mathcal{M}$, the \emph{cut time} of $\gamma_\xi$ is 
\[
    t_{\mathrm{c}}(\xi) := \sup\{t>0 \mid d(x,\gamma_\xi(t)) = t\}.
\]
When $t_{\mathrm{c}}(\xi)<\infty$, it is the last value of $t$ such that $\gamma_\xi|_{[0,t]}$ is minimal; see~\cite[III.4.1]{Sak96}. In other words, it is the time beyond which $\gamma_\xi$ ceases to be minimal. In that case, $t_{\mathrm{c}}(\xi)\xi$, resp.\ $\gamma_\xi(t_{\mathrm{c}}(\xi))$, is called the \emph{tangent cut point}, resp.\ \emph{cut point}, of $x$ along $\gamma_\xi$. The sets $\widetilde{C}_x := \{t_{\mathrm{c}}(\xi)\xi \mid \xi \in \mathrm{U}_x\mathcal{M}, \ t_{\mathrm{c}}(\xi)<+\infty\}$ and $C_x := \{\gamma_\xi(t_{\mathrm{c}}(\xi)) \mid \xi \in \mathrm{U}_x\mathcal{M}, \ t_{\mathrm{c}}(\xi)<+\infty\}$ are termed the \emph{tangent cut locus} and \emph{cut locus} of $x$, respectively. The cut time admits the following characterization, where a point $\Exp_x(t\xi)$ is said to be \emph{conjugate} to $x$ along the geodesic $t\mapsto\Exp_x(t\xi)$ if $t\xi$ is a critical point of $\Exp_x$, i.e., $\mathrm{D}\Exp_x(t\xi)$ is singular. 

\begin{theorem}[Lemma~5.2 in~\cite{CE75}, Proposition~4.1 in~\cite{Sak96}]  \label{thm:ab}
Let $\xi \in \mathrm{U}_x\mathcal{M}$ with $t_{\mathrm{c}}(\xi)<\infty$. Then $T = t_{\mathrm{c}}(\xi)$ 
if and only if the following holds for $t_0 = T$ and for no smaller value of $t_0$:
\\ \hspace*{2\parindent} (a) $\gamma_\xi(t_0)$ is a conjugate point of $x$ along $\gamma_\xi$;
\\ \hspace*{2\parindent} \textbf{or} (nonexclusive)
\\ \hspace*{2\parindent} (b) there exists $\zeta\in \mathrm{U}_x\mathcal{M}$, $\zeta \neq \xi$, such that $\gamma_\xi(t_0) = \gamma_\zeta(t_0)$.
\end{theorem}

The cut locus $C_x$ is closed~\cite[Proposition~5.4]{CE75}. The \emph{injectivity radius} of $x$ is 
\[
  \mathrm{inj}_{\mathcal{M}}(x) := d(x, C_x) = \inf\{t_{\mathrm{c}}(\xi) \mid \xi \in \mathrm{U}_x\mathcal{M}\}.  
\]
Equivalently, $\mathrm{inj}_{\mathcal{M}}(x)$ is the largest $r$ such that $\Exp_x$ is a diffeomorphism of the open ball of radius $r$ in $\mathrm{T}_x\mathcal{M}$ onto its image~\cite[Definition~5.5]{CE75}. 

The \emph{conjugate radius of $x$}, $\mathrm{conj}_{\mathcal{M}}(x)$, is the supremum of the radius of open balls centered at the origin of the tangent space $\mathrm{T}_x\mathcal{M}$ which contain no critical point of the exponential map $\Exp_x: \mathrm{T}_x\mathcal{M} \to \mathcal{M}$. 

It is known (see, e.g.,~\cite[eq.~(1.9)]{Xu2018} for the explicit statement, which follows from the work of Klingenberg~\cite{Kli82}) that
\begin{equation}  \label{eq:inj-ell-conj}
\mathrm{inj}_{\mathcal{M}}(x) = \min \{\frac12 \ell_{\mathcal{M}}(x), \mathrm{conj}_{\mathcal{M}}(x) \},
\end{equation}
where $\ell_{\mathcal{M}}(x)$ is the length of a shortest nontrivial geodesic loop at $x$. 
Finally, the \emph{injectivity radius of $\mathcal{M}$} is
\[
    \mathrm{inj}_{\mathcal{M}} := \inf \{\mathrm{inj}_{\mathcal{M}}(x) \mid x \in \mathcal{M} \}.
\]

\section{Background on the Stiefel manifold}
\label{sec:Stiefel}

In this section, we review the fundamentals of the Stiefel manifold and its family of Riemannian metrics introduced in~\cite{HuperMarkinaLeite2021}, with an emphasis on the Riemannian submersion structure that induces the metrics. We chiefly follow the notation of~\cite{ZimmermannHueper2022}. Background on manifolds and Lie groups can be found, e.g., in~\cite{Lee2003}. 

For $1\leq p \leq n$ with $n\geq2$,\footnote{This excludes the uninteresting case $p=n=1$ where $\mathrm{St}(n,p)=\{-1,1\}$.} the Stiefel manifold of orthonormal $p$-frames in $\mathbb{R}^n$ is
\begin{equation}  \label{eq:Stiefel}
\mathrm{St}(n,p) := \{X \in \mathbb{R}^{n\times p} \mid X^\top X = I_p \},    
\end{equation}
where $I_p$ is the identity matrix of size $p$. It is a submanifold of $\mathbb{R}^{n\times p}$; see, e.g.,~\cite[\S 3.3.2]{AMS2008}. 
Its tangent space at $X\in\mathrm{St}(n,p)$ is
\[
    \mathrm{T}_X \mathrm{St}(n,p) = \{\xi \in \mathbb{R}^{n\times p} \mid X^\top \xi + \xi^\top X = 0\}.
\]

The orthogonal group is denoted by 
\[
\mathrm{O}(n) := \{Q\in\mathbb{R}^{n\times n} \mid Q^\top Q = I_n\},     
\]
its connected component containing $I_n$ is the special orthogonal group denoted by $\mathrm{SO}(n)$,
and the set of skew-symmetric matrices of size $p$ is denoted by
\[
\mathcal{S}_{\mathrm{skew}}(p) := \{A\in\mathbb{R}^{p\times p} \mid A^\top = -A\}.    
\]
Let $X_\perp$ be such that $\begin{bmatrix} X & X_\perp \end{bmatrix}$ belongs to $\mathrm{O}(n)$. Then 
\begin{equation}  \label{eq:TStAH}
    \mathrm{T}_X \mathrm{St}(n,p) = \{X A + X_\perp H \mid A\in\mathcal{S}_{\mathrm{skew}}(p), H\in\mathbb{R}^{(n-p)\times p} \}.
\end{equation}

In the rest of this section, we restrict to $p\leq n-1$. However, in later sections, we sometimes refer to the Riemannian metrics~\eqref{eq:g} and~\eqref{eq:can-g} for $p=n$; this amounts to ignoring their $H$-term.

\subsection{A Riemannian-submersion route to the noncanonical metrics}
\label{sec:noncanonical}

Let $\beta>0$ and let us embark on a development that induces on the Stiefel manifold a Riemannian metric (\emph{metric} for short),~\eqref{eq:g}, parameterized by $\beta$. In this section~\ref{sec:noncanonical}, we exclude $\beta=\frac12$ to avoid a division by zero in~\eqref{eq:overline-g}; the case $\beta=\frac12$ (the canonical metric) is addressed in section~\ref{sec:canonical}.

Consider the Lie group action
\[
\tau : (\mathrm{O}(n) \times \mathrm{O}(p)) \times \mathrm{St}(n,p) \to \mathrm{St}(n,p):
\left( (Q,V), X \right) \mapsto QXV^\top =: \tau_X(Q,V).    
\]
Let $I_{n\times p} = \begin{bmatrix} I_p \\ 0_{(n-p)\times p} \end{bmatrix}$. 
The stabilizer of $I_{n\times p}$ is
\[
\tau_{I_{n\times p}}^{-1}(I_{n\times p}) = \left\{ \left( \begin{bmatrix} V & 0 \\ 0 & W \end{bmatrix}, V \right) \mid V \in \mathrm{O}(p), W \in \mathrm{O}(n-p) \right\} \simeq \mathrm{O}(n-p) \times \mathrm{O}(p).    
\]
Hence, as an application of~\cite[Proposition~A.2]{BendokatZimmermannAbsil2024}, the submanifold $\mathrm{St}(n,p)$ is diffeomorphic to the quotient manifold $(\mathrm{O}(n) \times \mathrm{O}(p)) / (\mathrm{O}(n-p) \times \mathrm{O}(p))$:
\[
    \mathrm{St}(n,p) 
    \simeq
    (\mathrm{O}(n) \times \mathrm{O}(p)) / (\mathrm{O}(n-p) \times \mathrm{O}(p)).
\]

As in~\cite{HuperMarkinaLeite2021}, consider on $\mathrm{O}(n) \times \mathrm{O}(p)$ the (pseudo-)Riemannian metric $\overline{g}$ defined by 
\begin{align}  
\overline{g}_{(Q,V)}\left( (Q\Omega_{\mathrm{a}},V\Psi_{\mathrm{a}}), (Q\Omega_{\mathrm{b}},V\Psi_{\mathrm{b}}) \right) & := \frac12 \mathrm{trace}(\Omega_{\mathrm{a}}^\top \Omega_{\mathrm{b}}) + \frac{1}{2} \frac{2\beta}{1-2\beta} \mathrm{trace} (\Psi_{\mathrm{a}}^\top \Psi_{\mathrm{b}})  \label{eq:overline-g}
\\ &= \frac12 \mathrm{trace}(\Omega_{\mathrm{a}}^\top \Omega_{\mathrm{b}}) + \frac{1}{2\alpha} \mathrm{trace} (\Psi_{\mathrm{a}}^\top \Psi_{\mathrm{b}}), \nonumber
\end{align}
where we have defined
\begin{equation}  \label{eq:alpha-beta}
    \alpha := \frac1{2\beta}-1,  
\quad \text{i.e.,} \quad \beta = \frac{1/2}{\alpha+1},
\end{equation}
to recover the parameterization used in~\cite{HuperMarkinaLeite2021}. An advantage of the $\beta$ parameterization, also used in~\cite{nguyen2021curvatures}, is that it simplifies several expressions, notably the Stiefel metric~\eqref{eq:g}. Observe that, in this section~\ref{sec:noncanonical}, the range of $\alpha$ is $(-1,\infty)\setminus\{0\}$. For $\alpha > 0$ (i.e., $\beta < \frac12$),~\eqref{eq:overline-g} defines a bona-fide Riemannian metric; for $\alpha < 0$ (i.e., $\beta > \frac12$), it is indefinite, making $\overline{g}$ a pseudo-Riemannian metric.
The factor of $\frac12$ in~\eqref{eq:overline-g} has been introduced so that, notably, the length of the induced geodesic loops on $\mathrm{St}(n,1)$, which is the unit sphere, have length $2\pi$ instead of the odd-looking $2\sqrt{2}\pi$. 

The Lie algebra of the stabilizer $\tau_{I_{n\times p}}^{-1}(I_{n\times p})$ is
\[
\mathfrak{k} = \left\{ \left( \begin{bmatrix} \Psi & 0 \\ 0 & C \end{bmatrix}, \Psi \right) \mid \Psi \in \mathcal{S}_{\mathrm{skew}}(p), C \in \mathcal{S}_{\mathrm{skew}}(n-p) \right\}.    
\]
Let $\mathfrak{o}(n) \times \mathfrak{o}(p) = \mathcal{S}_{\mathrm{skew}}(n) \times \mathcal{S}_{\mathrm{skew}}(p)$ denote the Lie algebra of $\mathrm{O}(n) \times \mathrm{O}(p)$. As a preparation for the forthcoming notion of horizontal space, which plays a central role in the (pseudo-)Riemannian submersion structure that we aim to emphasize, we let $\mathfrak{m} \subset \mathfrak{o}(n) \times \mathfrak{o}(p)$ be the orthogonal companion of $\mathfrak{k}$ with respect to $\overline{g}$, i.e., 
\[
\mathfrak{m} = \{(\Omega,\Psi) \in \mathfrak{o}(n) \times \mathfrak{o}(p) \mid \overline{g}_{(I_n,I_p)}\left((\Omega_{\mathrm{a}},\Psi_{\mathrm{a}}), (\Omega,\Psi)\right) = 0 \text{ for all } (\Omega_{\mathrm{a}},\Psi_{\mathrm{a}}) \in \mathfrak{k}\}.
\] 
Since the restriction of $\overline{g}$ to $\mathfrak{k}$ remains nondegenerate~\cite[Proposition~2]{HuperMarkinaLeite2021}, it follows that $\mathfrak{m}$ is a complement to $\mathfrak{k}$ in $\mathfrak{o}(n) \times \mathfrak{o}(p)$~\cite[Proposition~1]{HuperMarkinaLeite2021}. Let us obtain an explicit expression for $\mathfrak{m}$. For $(\Omega,\Psi) \in \mathfrak{o}(n) \times \mathfrak{o}(p)$, it holds that $(\Omega,\Psi) \in \mathfrak{m}$ if and only if $\frac12 \mathrm{trace}\left( \left[\begin{smallmatrix} \Psi_{\mathrm{a}} & 0 \\ 0 & C_{\mathrm{a}} \end{smallmatrix}\right]^\top \Omega \right) + \frac1{2\alpha} \mathrm{trace} \left( \Psi_{\mathrm{a}}^\top \Psi \right) = 0$ for all $\Psi_{\mathrm{a}} \in \mathcal{S}_{\mathrm{skew}}(p)$ and $C_{\mathrm{a}} \in \mathcal{S}_{\mathrm{skew}}(n-p)$. Decomposing $\Omega$ in blocks as $\left[\begin{smallmatrix} \Omega_{11} & -\Omega_{21}^\top \\ \Omega_{21} & \Omega_{22} \end{smallmatrix}\right]$, the condition reads 
$\mathrm{trace}\left( \Psi_{\mathrm{a}}^\top (\Omega_{11} + \frac1\alpha \Psi) \right) + \mathrm{trace}\left( C_{\mathrm{a}}^\top \Omega_{22} \right) = 0$ for all $\Psi_{\mathrm{a}} \in \mathcal{S}_{\mathrm{skew}}(p)$ and $C_{\mathrm{a}} \in \mathcal{S}_{\mathrm{skew}}(n-p)$. Since $\Omega_{11} + \frac1\alpha \Psi \in \mathcal{S}_{\mathrm{skew}}(p)$ and $\Omega_{22} \in \mathcal{S}_{\mathrm{skew}}(n-p)$, the condition is equivalent to $\Omega_{11} = -\frac1\alpha \Psi$ and $\Omega_{22} = 0$. We have thus obtained the expression
\begin{align}
    \mathfrak{m} &= \left\{ \left( \begin{bmatrix} -\tfrac1\alpha \tilde{\Psi} & -\tilde{B}^\top \\ \tilde{B} & 0 \end{bmatrix}, \tilde{\Psi} \right) \mid \tilde{\Psi} \in \mathcal{S}_{\mathrm{skew}}(p), \tilde{B} \in \mathbb{R}^{(n-p)\times p} \right\}  \label{eq:hor-alpha}
    \\ &= \left\{ \left( \begin{bmatrix} 2\beta \Psi & -B^\top \\ B & 0 \end{bmatrix}, -(1-2\beta) \Psi \right) \mid \Psi \in \mathcal{S}_{\mathrm{skew}}(p), B \in \mathbb{R}^{(n-p)\times p} \right\}.  \nonumber
\end{align}

Consider the map $\varphi = \tau_{I_{n\times p}}$, i.e.,
\[
    \varphi : (\mathrm{O}(n) \times \mathrm{O}(p)) \to \mathrm{St}(n,p):
    (Q,V) \mapsto QI_{n\times p}V^\top.    
\]
In the vocabulary of principal fiber bundles, the domain of $\varphi$ is termed the \emph{total space} and its co-domain the \emph{base space}. The \emph{fiber} above $\varphi(Q,V)$ is 
\begin{multline}  \label{eq:fiber}
\varphi^{-1}(\varphi(Q,V)) = (Q,V) \varphi^{-1}(I_{n\times p}) 
\\ = \left\{ \left( Q \begin{bmatrix} R_1 & 0 \\ 0 & R_2 \end{bmatrix}, V R_1 \right) \mid R_1 \in \mathrm{O}(p), R_2 \in \mathrm{O}(n-p) \right\}.     
\end{multline}
At $(Q,V)$, the tangent space to the fiber (the \emph{vertical space}) is $(Q,V)\mathfrak{k}$ and its orthogonal complement (the \emph{horizontal space}) is $(Q,V)\mathfrak{m}$. 

Let $\mathrm{D}\varphi(a;b)$ denote the derivative of $\varphi$ at $a$ along $b$. When $b$ is a horizontal direction, we obtain
\begin{multline*}
    \mathrm{D}\varphi\left(
    (Q,V); \left( Q \begin{bmatrix} 2\beta \Psi & -B^\top \\ B & 0 \end{bmatrix}, -V (1-2\beta) \Psi \right)
    \right)
\\    = Q \begin{bmatrix} 2\beta \Psi & -B^\top \\ B & 0 \end{bmatrix} I_{n\times p} V^\top + Q I_{n\times p} (1-2\beta) \Psi V^\top
    = Q \begin{bmatrix} \Psi \\ B \end{bmatrix} V^\top.
\end{multline*}
A classical result in (pseudo-)Riemannian submersion theory (see, e.g.,~\cite[\S 7.44]{ONe1983}) is that, for all $\xi\in \mathrm{T}_X \mathrm{St}(n,p)$ and all $(Q,V)$ in the fiber $\varphi^{-1}(X)$, there is a unique $\overline{\xi}_{(Q,V)}$ in the horizontal space $(Q,V)\mathfrak{m}$ such that $\mathrm{D}\varphi\left((Q,V); \overline{\xi}_{(Q,V)}\right) = \xi$. This $\overline{\xi}_{(Q,V)}$ is termed the \emph{horizontal lift} of $\xi$ at $(Q,V)$. Expressing $\xi$ as 
\[
\xi = Q \begin{bmatrix} V^T A \\ H \end{bmatrix} 
= Q \begin{bmatrix} V^T A V\\ H V \end{bmatrix} V^\top, 
\]
one readily obtains 
\begin{equation}  \label{eq:overline-xi}
\overline{\xi}_{(Q,V)} =  \left( Q \begin{bmatrix} 2\beta V^\top A V & -V^\top H^\top \\ H V & 0 \end{bmatrix}, -(1-2\beta) V V^\top A V \right).   
\end{equation}

The following invariance holds: given $\xi$ and $\zeta$ in $\mathrm{T}_X \mathrm{St}(n,p)$, the inner product $\overline{g}_{(Q,V)}(\overline{\xi}_{(Q,V)}, \overline{\zeta}_{(Q,V)})$ does not depend on the choice of $(Q,V)$ in the fiber $\varphi^{-1}(X)$, i.e.
\begin{equation}  \label{eq:lift-metric-invariance}
    \overline{g}_{(Q,V)}(\overline{\xi}_{(Q,V)}, \overline{\zeta}_{(Q,V)}) = \overline{g}_{(\tilde{Q},\tilde{V})}(\overline{\xi}_{(\tilde{Q},\tilde{V})}, \overline{\zeta}_{(\tilde{Q},\tilde{V})}) \quad \text{for all $(Q,V), (\tilde{Q},\tilde{V}) \in \varphi^{-1}(X)$}.
\end{equation} 
Hence $\overline{g}$ induces a well-defined (pseudo-)Riemannian metric $g$ on $\mathrm{St}(n,p)$. Taking into account that each fiber contains a point where the second component, $V$, is the identity, we obtain the following expression for $g$: for all $\xi_{\mathrm{a}} = Q \left[\begin{smallmatrix}A_{\mathrm{a}} \\ H_{\mathrm{a}}\end{smallmatrix}\right]$ and $\xi_{\mathrm{b}} = Q \left[\begin{smallmatrix}A_{\mathrm{b}} \\ H_{\mathrm{b}}\end{smallmatrix}\right]$, we have
\begin{equation}  \label{eq:g}
g_{Q I_{n\times p}}\left( \xi_{\mathrm{a}}, \xi_{\mathrm{b}} \right) 
:= \overline{g}_{(Q,I_p)}\left( \overline{\xi_{\mathrm{a}}}_{(Q,I_p)}, \overline{\xi_{\mathrm{b}}}_{(Q,I_p)} \right)
= \beta \, \mathrm{trace}(A_{\mathrm{a}}^\top A_{\mathrm{b}}) + \mathrm{trace}(H_{\mathrm{a}}^\top H_{\mathrm{b}}),     
\end{equation}
which turns out to be a bona-fide Riemannian metric (since we have assumed that $\beta > 0$) even though, as mentioned above, $\overline{g}$ is indefinite for $\beta > \frac12$.
The map $\varphi$ from $\mathrm{O}(n) \times \mathrm{O}(p)$ to $\mathrm{St}(n,p)$, respectively endowed with the (pseudo-)Riemannian metrics $\overline{g}$ and $g$, is termed a \emph{(pseudo-)Riemannian submersion}. Note that the terms pseudo-Riemannian and semi-Riemannian are used interchangeably in the literature, e.g.,~\cite{HuperMarkinaLeite2021,ONe1983}. In the rest of the paper, the notation $\mathrm{St}_\beta(n,p)$ is sometimes used to recall that the Riemannian metric~\eqref{eq:g} depends on $\beta$. 
Observe that~\eqref{eq:g} is the metric considered in~\cite[eq.~(5)]{ZimmermannHueper2022}. 

The invariance~\eqref{eq:lift-metric-invariance} is not coincidental. The metric $\overline{g}$ is bi-invariant, hence in particular the action of $\mathrm{O}(n-p) \times \mathrm{O}(p)$ on $\mathrm{O}(n) \times \mathrm{O}(p)$ by right multiplication is isometric, and the invariance result follows from~\cite[Theorem~23.14]{GallierQuaintance2020} using~\cite[Corollary~23.10]{GallierQuaintance2020}.

The geodesic in $\mathrm{O}(n) \times \mathrm{O}(p)$ (endowed with the (pseudo-)Riemannian metric $\overline{g}$) starting at $(Q,V)$ with initial velocity $(Q\Omega,V\Psi)$ is $t \mapsto \Exp_{(Q,V)}(t(Q\Omega,V\Psi))$ with
\[
\Exp_{(Q,V)}(Q\Omega,V\Psi) = (Q\exp_{\mathrm{m}}(\Omega), V\exp_{\mathrm{m}}(\Psi)),
\]  
where $\Exp$ in the left-hand side is the Riemannian exponential in the total space $\mathrm{O}(n) \times \mathrm{O}(p)$ and $\exp_{\mathrm{m}}$ denotes the matrix exponential; this follows from the bi-invariance of $\overline{g}$ (see~\cite[Proposition~11.9(6)]{ONe1983}).
Observe that $\beta$ does not affect the geodesics in $\mathrm{O}(n) \times \mathrm{O}(p)$, but it affects the notion of horizontality in view of~\eqref{eq:hor-alpha}. 

The (pseudo-)Riemannian submersion route to the metric~\eqref{eq:g} offers an appreciable benefit: the theory establishes that if a geodesic in the domain $\mathrm{O}(n) \times \mathrm{O}(p)$ of the (pseudo-)Riemannian submersion $\varphi$ is horizontal at some point, then it is always horizontal, and its image by $\varphi$ is a geodesic in the codomain $\mathrm{St}_\beta(n,p)$ 
(see~\cite{ONeill1967}, and~\cite[Theorem~4.4]{CaponioJavaloyesPiccione2010} for the generalization to pseudo-Riemannian submersions).
Hence, given $X\in\mathrm{St}(n,p)$ and $\xi\in\mathrm{T}_X\mathrm{St}(n,p)$, regardless of the choice of $(Q,V)$ in the fiber $\varphi^{-1}(X)$, the exponential on $\mathrm{St}(n,p)$ endowed with the metric $g$~\eqref{eq:g} is given by
\begin{equation}  \label{eq:Exp}
\Exp_X \xi = \varphi\left( \Exp_{(Q,V)}(\overline{\xi}_{(Q,V)}) \right)
= Q\exp_{\mathrm{m}}(\Omega) I_{n\times p} \exp_{\mathrm{m}}(-\Psi) V^\top,   
\end{equation}
where $\overline{\xi}_{(Q,V)} = (Q\Omega,V\Psi)$. This yields the expression
\begin{equation}  \label{eq:geod-Stiefel}
\Exp_{Q I_{n\times p}}\left( Q \begin{bmatrix}A\\H\end{bmatrix} \right) 
= Q \exp_{\mathrm{m}} \begin{bmatrix}2\beta A & -H^\top \\ H & 0\end{bmatrix} I_{n\times p} \exp_{\mathrm{m}}((1-2\beta)A),     
\end{equation}
which is seen to be equivalent to~\cite[eq.~(10)]{ZimmermannHueper2022}. 
As an aside, it is also possible to obtain~\eqref{eq:geod-Stiefel} by applying~\cite[Proposition~23.27]{GallierQuaintance2020} on $(\mathrm{O}(n) \times \mathrm{O}(p)) / (\mathrm{O}(n-p) \times \mathrm{O}(p))$, which is a naturally reductive homogeneous space in view of~\cite[Proposition~23.29]{GallierQuaintance2020}.

Observe from~\eqref{eq:geod-Stiefel} that the geodesics at $Q I_{n\times p}$ are the geodesics at $I_{n\times p}$ multiplied on the left by $Q$. It follows that the injectivity radius is the same at every point of $\mathrm{St}(n,p)$. In the rest of the paper, we single out the point $I_{n\times p}$.

As for the derivative of the Riemannian exponential, which is involved in the definition of conjugate points,~\eqref{eq:Exp} yields the relation
\[
    \mathrm{D}\Exp_X(\xi;\breve{\xi})
    = \mathrm{D}\varphi\left(\Exp_{(Q,V)}(\overline{\xi}_{(Q,V)}); \mathrm{D} \Exp_{(Q,V)}(\overline{\xi}_{(Q,V)}; \overline{\breve{\xi}}_{(Q,V)}) \right)
\]
for all $\xi\in \mathrm{T}_X\mathrm{St}(n,p)$   
and $\breve{\xi} \in \mathrm{T}_{\xi} \mathrm{T}_X \mathrm{St}(n,p) \simeq \mathrm{T}_X \mathrm{St}(n,p)$, where $\varphi(Q,V) = X$.
We will only need the following consequence, which we state in a way that only invokes basic concepts of matrix theory, analysis and differential geometry: 

\begin{proposition}  \label{prp:conj-Stiefel-gen}
Let the Stiefel manifold $\mathrm{St}(n,p)$ be endowed with the Riemannian metric~\eqref{eq:g} with $\beta\neq\frac12$. Let $X\in\mathrm{St}(n,p)$, $Q\in\mathrm{O}(n)$ such that  
$Q I_{n\times p} = X$, 
$\xi\in \mathrm{T}_X\mathrm{St}(n,p)$   
and $\breve{\xi} \in \mathrm{T}_{\xi} \mathrm{T}_X \mathrm{St}(n,p) \simeq \mathrm{T}_X \mathrm{St}(n,p)$. 
Represent $\xi$ and $\breve{\xi}$ as in~\eqref{eq:TStAH}, namely 
\[
\xi = Q \begin{bmatrix} A \\ H \end{bmatrix}, \quad \breve{\xi} = Q \begin{bmatrix} \breve{A} \\ \breve{H}  \end{bmatrix},    
\]
where $A,\breve{A}\in\mathcal{S}_{\mathrm{skew}}(p)$ and $H,\breve{H}\in\mathbb{R}^{(n-p)\times p}$.
Following~\eqref{eq:overline-xi}, let
\[
\Omega := \begin{bmatrix}2\beta A & -H^\top \\ H & 0\end{bmatrix}, \
\Psi := -(1-2\beta)A, \
\breve{\Omega} := \begin{bmatrix}2\beta \breve{A} & -\breve{H}^\top \\ \breve{H} & 0\end{bmatrix},
\breve{\Psi} := -(1-2\beta)\breve{A}.
\]    
Then 
\begin{equation}  \label{eq:noncan-der}
\mathrm{D}\Exp_X(t\xi;\breve{\xi}) = 
            \left.\frac{\mathrm{d}}{\mathrm{d}\epsilon}\right|_{\epsilon=0} Q \exp_{\mathrm{m}}(t\Omega + \epsilon \breve\Omega) 
            I_{n\times p} \exp_{\mathrm{m}}(-t\Psi - \epsilon \breve\Psi).   
\end{equation}
\end{proposition}

\subsection{The canonical case}
\label{sec:canonical}

The canonical metric on the Stiefel manifold~\cite{EAS98} is $g$~\eqref{eq:g} with $\beta = \frac12$. In view of~\eqref{eq:alpha-beta}, $\beta = \frac12$ corresponds to $\alpha = 0$, for which $\overline{g}$~\eqref{eq:overline-g} is no longer a bona-fide metric. The necessary facts for the canonical metric can be obtained by taking the limit as $\alpha\to0$. However, to err on the side of caution, we review a Riemannian submersion structure that induces the canonical metric. 

We opt for a fairly detailed overview of this Riemannian submersion structure (which is not mentioned in~\cite{AMS2008} and only briefly considered in~\cite[Example~9.36]{boumal2023intromanifolds}) because it was instrumental in our discovery of conjugate points. With the exception of Proposition~\ref{prp:can-conj-Stiefel-gen}, the ingredients can already be found in~\cite[\S 2.3.1]{EAS98}, \cite{HuperMarkinaLeite2021}, and~\cite[\S 23.7]{GallierQuaintance2020}.

Consider the Lie group action
\[
\sigma : \mathrm{O}(n) \times \mathrm{St}(n,p) \to \mathrm{St}(n,p):
\left( Q, X \right) \mapsto QX =: \sigma_X(Q).    
\]
The stabilizer of $I_{n\times p}$ is
\[
\sigma_{I_{n\times p}}^{-1}(I_{n\times p}) = \left\{ \begin{bmatrix} I & 0 \\ 0 & W \end{bmatrix} \mid W \in \mathrm{O}(n-p) \right\} \simeq \mathrm{O}(n-p).    
\]
Hence, as an application of~\cite[Proposition~A.2]{BendokatZimmermannAbsil2024}, the submanifold $\mathrm{St}(n,p)$ is diffeomorphic to the quotient manifold $\mathrm{O}(n) / \mathrm{O}(n-p)$:
\[
    \mathrm{St}(n,p) 
    \simeq
    \mathrm{O}(n) / \mathrm{O}(n-p).
\]

As in~\cite{HuperMarkinaLeite2021}, consider on $\mathrm{O}(n)$ the Riemannian metric $\overline{g}$ defined by 
\begin{equation}  \label{eq:can-overline-g}
\overline{g}_{Q}\left( Q\Omega_{\mathrm{a}}, Q\Omega_{\mathrm{b}} \right) := \frac12 \mathrm{trace}(\Omega_{\mathrm{a}}^\top \Omega_{\mathrm{b}}). 
\end{equation}
The Lie algebra of the stabilizer $\sigma_{I_{n\times p}}^{-1}(I_{n\times p})$ is
\[
\mathfrak{k} = \left\{ \begin{bmatrix} 0 & 0 \\ 0 & C \end{bmatrix} \mid C \in \mathcal{S}_{\mathrm{skew}}(n-p) \right\}.    
\]
Its complement with respect to $\overline{g}$ is 
\[
    \mathfrak{m} = \left\{ \begin{bmatrix} \Psi & -B^\top \\ B & 0 \end{bmatrix} \mid  \Psi \in \mathcal{S}_{\mathrm{skew}}(p), B \in \mathbb{R}^{(n-p)\times p} \right\}.
\]

Consider the map $\varphi = \sigma_{I_{n\times p}}$, i.e.,
\begin{equation}  \label{eq:can-varphi}
    \varphi : \mathrm{O}(n) \to \mathrm{St}(n,p):
    Q \mapsto QI_{n\times p}.    
\end{equation}
The fiber above $\varphi(Q)$ is 
\[
\varphi^{-1}(\varphi(Q)) = Q \varphi^{-1}(I_{n\times p}) 
= \left\{ Q \begin{bmatrix} I & 0 \\ 0 & R_2 \end{bmatrix} \mid R_2 \in \mathrm{O}(n-p) \right\}.     
\]
At $Q$, the tangent space to the fiber (the \emph{vertical space}) is $Q\mathfrak{k}$ and its orthogonal complement (the \emph{horizontal space}) is $Q\mathfrak{m}$. 

The horizontal lift of
\[
\xi = Q \begin{bmatrix} A \\ H \end{bmatrix} 
\]
is
\begin{equation}  \label{eq:can-overline-xi}
\overline{\xi}_{Q} =  Q \begin{bmatrix} A & -H^\top \\ H & 0 \end{bmatrix}.   
\end{equation}

The metric $\overline{g}$ induces a well-defined Riemannian metric $g$ on $\mathrm{St}(n,p)$: for all $\xi_{\mathrm{a}} = Q \left[\begin{smallmatrix}A_{\mathrm{a}} \\ H_{\mathrm{a}}\end{smallmatrix}\right]$ and $\xi_{\mathrm{b}} = Q \left[\begin{smallmatrix}A_{\mathrm{b}} \\ H_{\mathrm{b}}\end{smallmatrix}\right]$,
\begin{equation}  \label{eq:can-g}
g_{Q I_{n\times p}}\left( \xi_{\mathrm{a}}, \xi_{\mathrm{b}} \right) 
:= \overline{g}_{Q}\left( \overline{\xi_{\mathrm{a}}}_{Q}, \overline{\xi_{\mathrm{b}}}_{Q} \right)
= \frac12 \mathrm{trace}(A_{\mathrm{a}}^\top A_{\mathrm{b}}) + \mathrm{trace}(H_{\mathrm{a}}^\top H_{\mathrm{b}}).     
\end{equation}
The map $\varphi$ from $\mathrm{O}(n)$ to $\mathrm{St}(n,p)$, respectively endowed with the metrics $\overline{g}$ and $g$, is a Riemannian submersion.
This yields the expression
\begin{equation}  \label{eq:can-geod-Stiefel}
\Exp_{Q I_{n\times p}}\left( Q \begin{bmatrix}A\\H\end{bmatrix} \right) 
= Q \exp_{\mathrm{m}} \begin{bmatrix} A & -H^\top \\ H & 0\end{bmatrix} I_{n\times p},     
\end{equation}
which is indeed~\eqref{eq:geod-Stiefel} with $\beta=\frac12$.

Regarding the derivative of the Riemannian exponential, we have:
\begin{proposition}  \label{prp:can-conj-Stiefel-gen}
Let the Stiefel manifold $\mathrm{St}(n,p)$ be endowed with the canonical metric,~\eqref{eq:can-g}, i.e.,~\eqref{eq:g} with $\beta=\frac12$. Let $X\in\mathrm{St}(n,p)$, $Q\in\mathrm{O}(n)$ such that
$Q I_{n\times p} = X$, 
$\xi\in \mathrm{T}_X\mathrm{St}(n,p)$   
and $\breve{\xi} \in \mathrm{T}_{\xi} \mathrm{T}_X \mathrm{St}(n,p) \simeq \mathrm{T}_X \mathrm{St}(n,p)$. 
Represent $\xi$ and $\breve{\xi}$ as in~\eqref{eq:TStAH}, namely 
\[
\xi = Q \begin{bmatrix} A \\ H \end{bmatrix}, \quad \breve{\xi} = Q \begin{bmatrix} \breve{A} \\ \breve{H}  \end{bmatrix},    
\]
where $A,\breve{A}\in\mathcal{S}_{\mathrm{skew}}(p)$ and $H,\breve{H}\in\mathbb{R}^{(n-p)\times p}$.
Following~\eqref{eq:can-overline-xi}, let
\[
\Omega := \begin{bmatrix} A & -H^\top \\ H & 0\end{bmatrix}, \
\breve{\Omega} := \begin{bmatrix}\breve{A} & -\breve{H}^\top \\ \breve{H} & 0\end{bmatrix}.
\]    
Then 
\begin{equation}  \label{eq:can-der}
\mathrm{D}\Exp_X(t\xi;\breve{\xi}) =
            \left.\frac{\mathrm{d}}{\mathrm{d}\epsilon}\right|_{\epsilon=0} Q \exp_{\mathrm{m}}(t\Omega + \epsilon \breve\Omega) 
            I_{n\times p}.
\end{equation}
\end{proposition}

\section{Preparatory lemmas}  \label{sec:lemmas}

The next technical result emerged from insight gathered in~\cite[\S 5.2.1]{Ren2013} and further numerical experiments. It gives a nontrivial root to~\eqref{eq:can-der}.

\begin{lemma}  \label{lmm:DexpOmega}
Given a positive integer $m$, let $I_m$ denote the identity matrix of size $m$, $D\in\mathbb{R}^{m\times m}$ nonzero, and $w\in\mathbb{R}.$ 
Further let  
\[
    \Omega := \begin{bmatrix}0 & -I_m \\ I_m & 0\end{bmatrix} \quad 
    \text{and} \quad
    \breve\Omega := \begin{bmatrix} -w D & D \\ D & 0 \end{bmatrix}.
\]
Then 
\begin{equation}  \label{eq:DexpOmega}
\left.\frac{\mathrm{d}}{\mathrm{d}\epsilon}\right|_{\epsilon=0} \exp_{\mathrm{m}}(t\Omega + \epsilon \breve\Omega) 
I_{2m\times m} = 0   
\end{equation}
if and only if
\begin{subequations}  \label{eqs:tw}
\begin{equation}  \label{eq:tw-t}
    \frac{\sin t}{t} + \cos t = 0    
\end{equation}
and 
\begin{equation}
    w = \frac{2}{t}
\end{equation} 
\end{subequations}
or 
\begin{equation}  \label{eq:tw-alt}
t = k \pi,\ k\in\mathbb{Z}\setminus\{0\} 
\quad \text{and} \quad w = 0.
\end{equation}
In other words, the first $m$ columns of the derivative of the matrix exponential at $t\Omega$ along $\breve\Omega$ are zero if and only if~\eqref{eqs:tw} or~\eqref{eq:tw-alt} holds. In particular, the smallest positive $t$ that satisfies~\eqref{eq:DexpOmega} is the smallest positive solution of~\eqref{eq:tw-t}. 
\end{lemma}

\begin{proof}
    It is readily seen that neither~\eqref{eq:DexpOmega} nor~\eqref{eqs:tw} hold when $t=0$. It remains to consider the case $t\neq0$, which gives us license to divide by $t$ in forthcoming developments.

According to a classical formula for the derivative of the matrix exponential, long known in the physics literature~\cite{KarplusSchwinger1948} (or see~\cite{Mathias1992}), 
\begin{equation}  \label{eq:Dexp}
    \left.\frac{\mathrm{d}}{\mathrm{d}\epsilon}\right|_{\epsilon=0} \exp_{\mathrm{m}}(t\Omega + \epsilon \breve\Omega) = \int_0^1 \exp_{\mathrm{m}}((1-\sigma)t\Omega) \breve\Omega \exp_{\mathrm{m}}(\sigma t\Omega) \mathrm{d}\sigma.
\end{equation}
In view of the formula (which follows from~\cite[\S 2.1]{GallierQuaintance2020}) 
\begin{equation}  \label{eq:expstO}
\exp_{\mathrm{m}}(\sigma t \Omega) = \begin{bmatrix} \cos(\sigma t) I_m & -\sin(\sigma t) I_m \\ \sin(\sigma t) I_m & \cos(\sigma t)I_m \end{bmatrix},     
\end{equation}
the top-left $m\times m$ block of~\eqref{eq:Dexp} expands as
\[
\int_0^1 \left(
    -w \cos(1-\sigma)t \cos \sigma t - \sin (1-\sigma)t \cos \sigma t + \cos(1-\sigma)t \sin \sigma t
    \right) \mathrm{d}\sigma \ D,
\]
which, by product-to-sum trigonometric identities, is equal to
\[
    \frac12 \int_0^1 \left(
        -w \cos(1-2\sigma)t - w \cos t - \sin t - \sin(1-2\sigma) t + \sin t - \sin (1-2\sigma)t
        \right) \mathrm{d}\sigma \ D. 
\]
This further simplifies to 
\begin{equation}  \label{eq:tr}
-\frac{w}2 \left(\frac{\sin t}{t} + \cos t\right) D.    
\end{equation}

As for the bottom-left $m\times m$ block of~\eqref{eq:Dexp}, it expands as
\[
\int_0^1 \left( 
    -w \sin(1-\sigma)t \cos \sigma t + \cos (1-\sigma)t \cos \sigma t + \sin(1-\sigma)t \sin \sigma t
    \right) \mathrm{d}\sigma \ D,
\]
which, by a similar route, reduces to
\begin{equation}  \label{eq:omega}
    \frac12 \left(-w + \frac2t\right) \sin t \ D.
\end{equation}

Recall that $D$ is nonzero. The first $m$ columns of~\eqref{eq:Dexp} are thus zero if and only if both~\eqref{eq:tr} and~\eqref{eq:omega} equal zero, i.e.,~\eqref{eqs:tw} or~\eqref{eq:tw-alt} holds.
\end{proof}

The next generalization readily follows. It gives a nontrivial root to~\eqref{eq:noncan-der} for $\beta\neq1$, i.e., $\alpha\neq-\frac12$. 

\begin{lemma}  \label{lmm:DexpOmegaPsi}
Let $m$, $D$, $\Omega$ and $\breve\Omega$ be as in Lemma~\ref{lmm:DexpOmega}. Further let $\alpha\neq-\frac12$, 
\[
    \Psi := 0 \quad \text{and} \quad \breve\Psi := \alpha w D.    
\] 
Then
\begin{equation}  \label{eq:DexpIexp}
\left.\frac{\mathrm{d}}{\mathrm{d}\epsilon}\right|_{\epsilon=0} \exp_{\mathrm{m}}(t\Omega + \epsilon \breve\Omega) 
I_{2m\times m} \exp_{\mathrm{m}}(-t\Psi - \epsilon \breve\Psi) = 0   
\end{equation}
if and only if
\begin{subequations}  \label{eqs:twalpha}
\begin{equation}  \label{eq:talpha}
    \frac{\sin t}{t} + (1+2\alpha) \cos t = 0    
\end{equation}     
and 
\begin{equation}  \label{eq:w}
    w = \frac{2}{1+2\alpha} \frac{1}{t}
\end{equation} 
\end{subequations}
or 
\begin{equation}  \label{eq:twalpha-alt}
t = k \pi,\ k\in\mathbb{Z}\setminus\{0\} 
\quad \text{and} \quad w = 0.
\end{equation}
In particular, the smallest positive $t$ that satisfies~\eqref{eq:DexpIexp} is the smallest positive solution of~\eqref{eq:talpha} when $\alpha > -\frac12$ (i.e., $0 < \beta < 1$) and $\pi$ when $-1 < \alpha < -\frac12$ (i.e., $\beta > 1$).
\end{lemma}

\begin{proof}
The case $t=0$ goes as in the proof of Lemma~\ref{lmm:DexpOmega}. For the case $t\neq0$, by the product rule, 
\begin{multline*}
    \left.\frac{\mathrm{d}}{\mathrm{d}\epsilon}\right|_{\epsilon=0} \exp_{\mathrm{m}}(t\Omega + \epsilon \breve\Omega) 
    I_{2m\times m} \exp_{\mathrm{m}}(-t\Psi - \epsilon \breve\Psi)
    \\ = \left(\left.\frac{\mathrm{d}}{\mathrm{d}\epsilon}\right|_{\epsilon=0} \exp_{\mathrm{m}}(t\Omega + \epsilon \breve\Omega) 
    I_{2m\times m}\right)
    - 
    \exp_{\mathrm{m}}(t\Omega) I_{2m\times m} \breve\Psi.
\end{multline*}
Lemma~\ref{lmm:DexpOmega} gives an expression for the first term and, using~\eqref{eq:expstO} for the second term, the whole expression is found to reduce to
\[
\begin{bmatrix}
    -\frac{w}{2} \left(\frac{\sin t}{t} + \cos t + 2\alpha \cos t\right) D
    \\
    \frac12 \left( -w + \frac2t - 2\alpha w \right) \sin t \, D
\end{bmatrix}.    
\]
This is zero if and only if~\eqref{eqs:twalpha} or~\eqref{eq:twalpha-alt} holds.

To prove the last sentence of the lemma, let us analyze the sign of the left-hand side of~\eqref{eq:talpha} for $t,\beta>0$, which is equal to the sign of
\begin{equation}  \label{eq:F-beta-t}
    F(\beta,t) := \sin t + \frac{1-\beta}{\beta} t \cos t.
\end{equation}
First consider $0 < \beta < 1$; then $F(\beta,t) > 0$ for all $t\in(0,\frac\pi2]$ and $F(\beta,\pi) = -\frac{1-\beta}{\beta} \pi < 0$, thus~\eqref{eq:talpha} has a solution in $(\frac\pi2,\pi)$. Now consider $\beta \geq 1$. Then $F(\beta,t) > 0$ for all $t\in[\frac\pi2,\pi)$. Furthermore, $F(\beta,0) = 0$ and, for all $t\in(0,\frac\pi2)$, it holds that $\partial_t F(\beta,t) = \cos t + \frac{1-\beta}{\beta} (\cos t - t \sin t) = \frac1\beta \cos t - \frac{1-\beta}{\beta} t \sin t > 0$, which implies that $F(\beta,t) > 0$ for all $t\in(0,\frac\pi2)$. Hence $F(\beta,t)$ has no root $t$ in $(0,\pi)$, thus~\eqref{eq:talpha} has no solution in $[0,\pi)$.
\end{proof}

The following final preparatory lemma can be interpreted as the limit of the latter when $\alpha\to-\frac12$. It gives a nontrivial root to~\eqref{eq:noncan-der} for the remaining case: $\beta=1$, i.e., $\alpha=-\frac12$.

\begin{lemma}  \label{lmm:DexpOmegaPsi-Euc}
    Let $m$, $D$, $\Omega$, and $\Psi$ be as in Lemma~\ref{lmm:DexpOmegaPsi}. Now let $\alpha=-\frac12$,
    \[
        \breve\Omega := \begin{bmatrix} D & 0 \\ 0 & 0 \end{bmatrix} \quad \text{and} \quad \breve\Psi := -\alpha D.    
    \] 
    Then
    \[
    \left.\frac{\mathrm{d}}{\mathrm{d}\epsilon}\right|_{\epsilon=0} \exp_{\mathrm{m}}(t\Omega + \epsilon \breve\Omega) 
    I_{2m\times m} \exp_{\mathrm{m}}(-t\Psi - \epsilon \breve\Psi) = 0   
    \]
    if and only if
    \begin{equation}  \label{eq:tEuc}
        \frac{\sin t}{t} = 0,    
    \end{equation}
    i.e., $t = k\pi$, $k\in\mathbb{Z}\setminus\{0\}$.
    Observe that the smallest positive solution of~\eqref{eq:tEuc} is $\pi$.      
\end{lemma}
\begin{proof}
The case $t=0$ goes as in the previous proofs. For the case $t\neq 0$, proceeding as in the previous results, one obtains that 
\[
    \left.\frac{\mathrm{d}}{\mathrm{d}\epsilon}\right|_{\epsilon=0} \exp_{\mathrm{m}}(t\Omega + \epsilon \breve\Omega) 
    I_{2m\times m} \exp_{\mathrm{m}}(-t\Psi - \epsilon \breve\Psi)
    =
    \begin{bmatrix} \frac{\sin t}{2t} D
        \\ 0_{m\times m}
    \end{bmatrix},
\]
and the result follows.
\end{proof}

\section{Conjugate points}
\label{sec:conjugate}

Consider the Stiefel manifold~\eqref{eq:Stiefel} 
endowed with the metric $g$~\eqref{eq:g}.
In this section, we restrict to the case $2 \leq p \leq n-2$, where the preparatory lemmas of section~\ref{sec:lemmas} can be leveraged to find conjugate points. Specifically, we combine Proposition~\ref{prp:conj-Stiefel-gen} and Proposition~\ref{prp:can-conj-Stiefel-gen} with the lemmas of section~\ref{sec:lemmas} to obtain conjugate points on the Stiefel manifold. The length of the geodesic along which those points are conjugate is therefore an upper bound on the conjugate radius of the Stiefel manifold.

\begin{theorem}  \label{thm:conjugate}
Let $2 \leq p \leq n-2$ and consider  $\mathrm{St}_\beta(n,p)$, the Stiefel manifold~\eqref{eq:Stiefel} endowed with the Riemannian metric $g$ as in~\eqref{eq:g} parameterized by $\beta>0$. Let $\alpha = \frac1{2\beta}-1$ as in~\eqref{eq:alpha-beta}. Let $t^{\mathrm{r}}_\beta$ denote the smallest positive solution of~\eqref{eq:talpha}, namely
\[
    t^{\mathrm{r}}_\beta := \min\left\{t>0 \mid \frac{\sin t}{t} + (1+2\alpha) \cos t = 0 \right\}.   
\]
Let $Q\in\mathrm{O}(n)$ and consider $\xi = Q \left[\begin{smallmatrix}A\\H\end{smallmatrix}\right] \in \mathrm{T}_{Q I_{n\times p}} \mathrm{St}(n,p)$ with 
\[
    A := 0_{p\times p} \quad \text{and} \quad H := \begin{bmatrix} I_2 & 0_{2\times(p-2)} \\ 0_{(n-p-2)\times 2} & 0_{(n-p-2)\times(p-2)} \end{bmatrix}.
\] 
Let $\gamma$ denote the geodesic starting at $Q I_{n\times p}$ in the direction of $\xi$, namely, in view of~\eqref{eq:geod-Stiefel}, $\gamma(t) = Q \exp_{\mathrm{m}}(t\Omega) I_{n\times p} \exp_{\mathrm{m}}(-t\Psi) = Q \exp_{\mathrm{m}}(t\Omega) I_{n\times p}$ where
\[
    \Omega := \begin{bmatrix}2\beta A & -H^\top \\ H & 0\end{bmatrix}
\quad \text{and} \quad
\Psi := -(1-2\beta) A = 0_{p\times p}.
\]
Then $\gamma(0) = Q I_{n\times p}$ and $\gamma(t^{\mathrm{r}}_\beta)$ are conjugate along $\gamma$. Moreover, $\gamma(0)$ and $\gamma(\pi)$ are conjugate along $\gamma$.

Furthermore, for all $T>0$, the length (in the sense of~\eqref{eq:g}) of $\gamma$ between $t=0$ and $t=T$ is $T\sqrt2$. Thus
\begin{equation}  \label{eq:conjugate}
\mathrm{conj}_{\mathrm{St}_\beta(n,p)} \leq \min\{t^{\mathrm{r}}_\beta,\pi\}\sqrt2,
\end{equation}
where the $\min$ is achieved by $t^{\mathrm{r}}_\beta$ when $\beta \leq 1$ and $\pi$ when $\beta \geq 1$.
\end{theorem}
\begin{proof}
We first prove the conjugacy of $\gamma(0)$ and $\gamma(t^{\mathrm{r}}_\beta)$ under the assumptions of the theorem, distinguishing three cases.

\underline{Case~1}: non-Euclidean noncanonical case, i.e., $\alpha \notin \{-\frac12,0\}$ or equivalently $\beta \notin \{1,\frac12\}$. Let $D\in\mathcal{S}_{\mathrm{skew}}(2)$ be nonzero, $w := \frac{2}{1+2\alpha} \frac{1}{t^{\mathrm{r}}_\beta}$ according to~\eqref{eq:w}, and $\breve\xi := Q \left[\begin{smallmatrix} \breve{A} \\ \breve{H} \end{smallmatrix}\right]$ with
\[
\breve{A} := \begin{bmatrix} -\frac{w}{2\beta} D & 0_{2\times(p-2)} \\ 0_{(p-2)\times2} & 0_{(p-2)\times(p-2)} \end{bmatrix}, 
\quad \breve{H} := \begin{bmatrix} D & 0_{2\times(p-2)} \\ 0_{(n-p-2)\times2} & 0_{(n-p-2)\times(p-2)} \end{bmatrix}.
\]
Based on Proposition~\ref{prp:conj-Stiefel-gen} and Lemma~\ref{lmm:DexpOmegaPsi}, we now prove that $\mathrm{D}\Exp_X(t^{\mathrm{r}}_\beta \xi;\breve{\xi}) = 0$, which is the conjugacy result. 

Define $\breve\Omega$ and $\breve\Psi$ as in Proposition~\ref{prp:conj-Stiefel-gen}. Let $\Omega_4$, $\Psi_2$, $\breve\Omega_4$, and $\breve\Psi_2$ (where the subscript recalls the dimension of the matrices) denote the like-named matrices of Lemma~\ref{lmm:DexpOmegaPsi}, where $m=2$ and $D$ and $w$ are those introduced above. Let $P$ be the $n\times n$ permutation matrix that shifts elements $p+1$ and $p+2$ up to elements 3 and 4. Then 
\[
P \Omega P^\top = \begin{bmatrix} \Omega_4 & 0_{4\times(n-4)} \\ 0_{(n-4)\times4} & 0_{(n-4)\times(n-4)} \end{bmatrix} =: \Omega_P,
\]
i.e., $\Omega$ is the $\Omega_4$ matrix of Lemma~\ref{lmm:DexpOmegaPsi} after padding with zeros and permutation. Likewise, 
\[ 
P \breve\Omega P^\top = \begin{bmatrix} \breve\Omega_4 & 0_{4\times(n-4)} \\ 0_{(n-4)\times4} & 0_{(n-4)\times(n-4)} \end{bmatrix} =: \breve\Omega_P. 
\]
By Proposition~\ref{prp:conj-Stiefel-gen}, letting ``$*$'' stand for an irrelevant constant, we have
\begin{align*}
&\mathrm{D}\Exp_X(t^{\mathrm{r}}_\beta \xi;\breve{\xi})
\\ &= \left.\frac{\mathrm{d}}{\mathrm{d}\epsilon}\right|_{\epsilon=0} Q \exp_{\mathrm{m}}(t^{\mathrm{r}}_\beta \Omega + \epsilon \breve\Omega) I_{n\times p} \exp_{\mathrm{m}}(-t^{\mathrm{r}}_\beta\Psi - \epsilon \breve\Psi)
\\ &= \left.\frac{\mathrm{d}}{\mathrm{d}\epsilon}\right|_{\epsilon=0} Q P^\top \exp_{\mathrm{m}}(t^{\mathrm{r}}_\beta \Omega_P + \epsilon \breve\Omega_P) P I_{n\times p} \exp_{\mathrm{m}}(-t^{\mathrm{r}}_\beta\Psi - \epsilon \breve\Psi)
\\ &= \left.\frac{\mathrm{d}}{\mathrm{d}\epsilon}\right|_{\epsilon=0} Q P^\top 
\begin{bmatrix} \exp_{\mathrm{m}}(t^{\mathrm{r}}_\beta \Omega_4 + \epsilon \breve\Omega_4) & 0_{4\times(n-4)} \\ 0_{(n-4)\times4} & I_{n-4} \end{bmatrix}
\\ & \qquad \begin{bmatrix} I_{4\times2} & 0_{4\times(p-2)} \\ 0_{(n-4)\times2} & * \end{bmatrix} 
\begin{bmatrix} \exp_{\mathrm{m}}(-t^{\mathrm{r}}_\beta\Psi_2 - \epsilon \breve\Psi_2) & 0_{2\times(p-2)} \\ 0_{2\times(p-2)} & I_{p-2} \end{bmatrix}
\\ &= \left.\frac{\mathrm{d}}{\mathrm{d}\epsilon}\right|_{\epsilon=0} Q P^\top 
\begin{bmatrix} \exp_{\mathrm{m}}(t^{\mathrm{r}}_\beta \Omega_4 + \epsilon \breve\Omega_4) & 0_{4\times(n-4)} \\ 0_{(n-4)\times4} & I_{n-4} \end{bmatrix}
\\ & \qquad \begin{bmatrix} I_{4\times2} \exp_{\mathrm{m}}(-t^{\mathrm{r}}_\beta\Psi_2 - \epsilon \breve\Psi_2) & 0_{4\times(p-2)} \\ 0_{(n-4)\times2} & * \end{bmatrix}
\\ &= \left.\frac{\mathrm{d}}{\mathrm{d}\epsilon}\right|_{\epsilon=0} Q P^\top 
\begin{bmatrix} \exp_{\mathrm{m}}(t^{\mathrm{r}}_\beta \Omega_4 + \epsilon \breve\Omega_4) I_{4\times2} \exp_{\mathrm{m}}(-t^{\mathrm{r}}_\beta\Psi_2 - \epsilon \breve\Psi_2) & 0_{4\times(p-2)} \\ 0_{(n-4)\times2} & * \end{bmatrix}
\\ &= 0,
\end{align*}
where the last equality follows from Lemma~\ref{lmm:DexpOmegaPsi}.

\underline{Case~2}: Euclidean case, i.e., $\alpha = -\frac12$ or equivalently $\beta = 1$. Let $D\in\mathcal{S}_{\mathrm{skew}}(2)$ be nonzero and
\[
\breve\xi := Q \begin{bmatrix} \breve{A} \\ \breve{H} \end{bmatrix}, \quad \text{with }
\breve{A} := \begin{bmatrix} \frac{1}{2} D & 0_{2\times(p-2)} \\ 0_{(p-2)\times2} & 0_{(p-2)\times(p-2)} \end{bmatrix}, 
\breve{H} := 0_{(n-p)\times p}.
\]
A development akin to Case~1, invoking Proposition~\ref{prp:conj-Stiefel-gen} and this time Lemma~\ref{lmm:DexpOmegaPsi-Euc}, again yields $\mathrm{D}\Exp_X(t^{\mathrm{r}}_\beta \xi;\breve{\xi}) = 0$.

\underline{Case~3}: canonical case, i.e., $\alpha = 0$ or equivalently $\beta = \frac12$. Let $\breve\xi$, $\breve{A}$, and $\breve{H}$ be as in Case~1. Here again, a development akin to Case~1, invoking this time Proposition~\ref{prp:can-conj-Stiefel-gen} and Lemma~\ref{lmm:DexpOmega}, yields $\mathrm{D}\Exp_X(t^{\mathrm{r}}_\beta \xi;\breve{\xi}) = 0$. 

This completes the proof of the conjugacy of $\gamma(0)$ and $\gamma(t^{\mathrm{r}}_\beta)$. The conjugacy of $\gamma(0)$ and $\gamma(\pi)$ is proved similarly, with $w$ replaced by $0$ and $t^{\mathrm{r}}_\beta$ by $\pi$, and omitting the redundant Case~2.

Finally, for all $T>0$, the length of $\gamma$ between $t=0$ and $t=T$ is $T \|\xi\|_{Q I_{n\times p}}$ with $\|\xi\|_{Q I_{n\times p}}^2 = g_{Q I_{n\times p}}(\xi,\xi) = \mathrm{trace}(H^\top H) = \|H\|_{\mathrm{F}}^2 = 2$. The bound~\eqref{eq:conjugate} on the conjugate radius is then direct. Which of $t^{\mathrm{r}}_\beta$ and $\pi$ achieves the $\min$ follows from the sign analysis in the proof of Lemma~\ref{lmm:DexpOmegaPsi}.
\end{proof}

\section{Geodesic loops}
\label{sec:loops}

In view of~\eqref{eq:inj-ell-conj}, upper bounds on the length of the shortest nontrivial geodesic loops also yield upper bounds on the injectivity radius.

We continue to consider $\mathrm{St}_\beta(n,p)$, the Stiefel manifold~\eqref{eq:Stiefel} endowed with the metric $g$~\eqref{eq:g} parameterized by $\beta>0$. We distinguish the following (overlapping) cases.

\subsection{Case $p=1$} Regardless of $\beta$, the Stiefel manifold reduces to the unit sphere as a Riemannian submanifold of the Euclidean space $\mathbb{R}^n$. Every geodesic loop has length $2\pi$. Hence $\ell_{\mathrm{St}_\beta(n,1)}(X) = 2\pi$ for all $X$. 

\subsection{Case $p=n \geq 2$} The Stiefel manifold reduces to the orthogonal group $\mathrm{O}(n)$ endowed with the Frobenius metric scaled by $\beta$. The shortest geodesics correspond to Givens rotations,
and their length is $\sqrt{2\beta}\,2\pi$; 
the result readily follows from Theorem~\ref{thm:loop-can} below since, when $p=n$, the Frobenius metric scaled by $\beta$ is the canonical metric scaled by $2\beta$.
Hence $\ell_{\mathrm{St}_\beta(n,n)}(X) = \sqrt{2\beta}\,2\pi$ for all $X$. 

\subsection{Case $\beta=\frac12$ (canonical metric)} It is known that $\ell_{\mathrm{St}_{\frac12}(n,p)}(X) = 2\pi$ for all $X$~\cite[p.~94]{Ren2013}. We give a more detailed proof for completeness.
\begin{theorem}[Rentmeesters~\cite{Ren2013}]  \label{thm:loop-can}
    On $\mathrm{St}_\beta(n,p)$ with $\beta=\frac{1}{2}$, the shortest geodesic loops have length $2\pi$.
\end{theorem}

\begin{proof}
Let $[0,1]\ni t\mapsto \Exp_{I_{n\times p}}\left(t \left[\begin{smallmatrix}A\\H\end{smallmatrix}\right] \right) = \exp_{\mathrm{m}} \left( t \left[\begin{smallmatrix} A & -H^\top \\ H & 0\end{smallmatrix}\right] \right) I_{n\times p}$ be a nontrivial geodesic loop, where the equality follows from~\eqref{eq:can-geod-Stiefel}. Since it is nontrivial, the initial velocity is nonzero, i.e., $\left[\begin{smallmatrix} A & -H^\top \\ H & 0\end{smallmatrix}\right] I_{n\times p} \neq 0$, and thus there is $\ell \in \{1,\dots,p\}$ such that $\left[\begin{smallmatrix} A & -H^\top \\ H & 0\end{smallmatrix}\right] e_\ell \neq 0$; and since it is a loop, it holds that
$I_{n\times p} = \exp_{\mathrm{m}} \left[\begin{smallmatrix} A & -H^\top \\ H & 0\end{smallmatrix}\right] I_{n\times p}$, and thus $e_\ell = \exp_{\mathrm{m}} \left[\begin{smallmatrix} A & -H^\top \\ H & 0\end{smallmatrix}\right] e_\ell$. 
Let $\left[\begin{smallmatrix} A & -H^\top \\ H & 0\end{smallmatrix}\right] = S \Lambda S^*$ be an eigenvalue decomposition. The matrix is skew-symmetric, hence $\Lambda = \mathrm{diag}(i\theta_1,\dots,i\theta_n)$ where the nonzero $i\theta$'s appear in complex conjugate pairs. Let $\mathcal{K}$ be the set of indices of the nonzero components of $S^* e_\ell$. 
We have $S \Lambda S^* e_\ell \neq 0$, 
thus there is $k\in\mathcal{K}$ such that $\theta_k\neq0$; and 
we have $e_\ell = \exp_\mathrm{m}(S \Lambda S^*) e_\ell$, hence $S^* e_\ell = \exp_{\mathrm{m}}(\Lambda) S^* e_\ell$, thus $e^{i\theta_k}=1$ for all $k\in\mathcal{K}$. 
Consequently, the length of the geodesic loop (in the canonical metric) is
$\frac{\sqrt{2}}{2}\left\|\left[\begin{smallmatrix}
    A&-H^T\\
    H&0
\end{smallmatrix}\right]\right\|_{\mathrm{F}}
= \frac{\sqrt{2}}{2}\sqrt{\sum_{j=1}^n \theta_k^2}
\geq \frac{\sqrt{2}}{2}\sqrt{(2\pi)^2 + (-2\pi)^2} = 2\pi$. Finally, since there exist geodesic loops of length $2\pi$, it follows that $2\pi$ is the length of the shortest geodesic loops. Note that the proof still holds when $p=n$: $H$ disappears and $\left[\begin{smallmatrix} A & -H^\top \\ H & 0\end{smallmatrix}\right]$ reduces to $A$.
\end{proof}

\subsection{Case $2 \leq p \leq n-1$ and $\beta\neq\frac12$ (noncanonical metrics)} To our knowledge, the shortest geodesic loops are unknown, except for the case $\beta = 1$ where these loops have length $2\pi$, as shown in the recent eprint~\cite{zimmermann2024injectivity}.
However, some geodesic loops are known, yielding an upper bound. Let $E_{i,j}$ denote the matrix of size $n\times p$ with $1$ at position $(i,j)$ and zeros elsewhere. In view of~\eqref{eq:geod-Stiefel}, the geodesic at $I_{n\times p}$ along $\xi = 2\pi (E_{2,1} - E_{1,2})$ makes a loop in unit time, and its length according to~\eqref{eq:g} is $\sqrt{8 \beta \pi^2} = \sqrt{2\beta} \, 2\pi$. Likewise, the geodesic at $I_{n\times p}$ along $\xi = 2\pi E_{n,1}$ makes a loop in unit time, and its length according to~\eqref{eq:g} is $2\pi$. Hence 
\begin{equation}  \label{eq:ell}
\ell_{\mathrm{St}_\beta(n,p)}(X) \leq \min\{\sqrt{2\beta},1\} 2\pi.    
\end{equation}

\section{Upper bounds on the injectivity radius}
\label{sec:inj-upper}

We can now exploit the knowledge on $\mathrm{conj}_{\mathrm{St}_\beta(n,p)}$ and $\ell_{\mathrm{St}_\beta(n,p)}$ gathered respectively in section~\ref{sec:conjugate} and~\ref{sec:loops} in order to produce an upper bound on the injectivity radius of the Stiefel manifold $\mathrm{St}_\beta(n,p)$.

\subsection{Case $p=1$ or $p\geq n-1$}
\label{sec:extreme-p}

As already mentioned, when $p=1$, regardless of $\beta$, the Stiefel manifold reduces to the unit sphere as a Riemannian submanifold of the Euclidean space $\mathbb{R}^n$. The injectivity radius is $\pi$; see, e.g.,~\cite[\S 1.6]{CE75}.  

When $p=n$, the Stiefel manifold reduces to the orthogonal group $\mathrm{O}(n)$ endowed with the Frobenius metric scaled by $\beta$. The injectivity radius is $\sqrt{2\beta}\pi$; this can be deduced from~\cite[Corollary~2.1]{Mataigne2023}. 

When $p=n-1$ and $\beta=\frac12$ (canonical metric), since the fibers of $\varphi$~\eqref{eq:can-varphi} in $\mathrm{SO}(n)$ are singletons, the injectivity radius of the Stiefel manifold is the same as the injectivity radius of $\mathrm{O}(n)$, namely $\pi$.

Finally, when $p=n-1$ and $\beta\neq\frac12$ (noncanonical metrics) with $p\geq2$ (the case $p=1$ has already been handled above), in view of~\eqref{eq:ell} and~\eqref{eq:inj-ell-conj}, we have
\begin{equation}  \label{eq:inj-St-n-1}
\mathrm{inj}_{\mathrm{St}(n,n-1)} \leq \min\{\sqrt{2\beta},1\} \pi.    
\end{equation}

\subsection{Case $2 \leq p \leq n-2$}

\begin{theorem}  \label{thm:inj}
    Let $2 \leq p \leq n-2$ and consider  $\mathrm{St}_\beta(n,p)$, the Stiefel manifold~\eqref{eq:Stiefel} endowed with the Riemannian metric $g$~\eqref{eq:g} parameterized by $\beta>0$.
Let $t^{\mathrm{r}}_\beta$ denote the smallest positive solution of~\eqref{eq:talpha}, namely, in view of~\eqref{eq:alpha-beta},
    \[
        t^{\mathrm{r}}_\beta 
        = \min \left\{t>0 \mid \frac{\sin t}{t} + \frac{1-\beta}{\beta} \cos t = 0 \right\}.   
    \]
Then 
\begin{equation}  \label{eq:inj}
\mathrm{inj}_{\mathrm{St}_\beta(n,p)} \leq \min\{ \sqrt{2\beta}\,\pi, \pi, t^{\mathrm{r}}_\beta\sqrt2 \} \quad =: \hat{\imath}_\beta.    
\end{equation}
\end{theorem}
\begin{proof}
Combine~\eqref{eq:inj-ell-conj} with~\eqref{eq:conjugate} and~\eqref{eq:ell}.
\end{proof}

Let $\beta_1 \approx 0.34689870829737$ denote the smallest root of $\frac{\sin(\sqrt\beta \pi)}{\sqrt\beta \pi} + \frac{1-\beta}{\beta} \cos(\sqrt\beta \pi)$ and let 
$\beta_2 = \left(1-\frac{\sqrt2}{\pi} \tan\frac{\pi}{\sqrt2}\right)^{-1} \approx 0.62839259859361$. An analysis of $\hat{\imath}_\beta$---see Figure~\ref{fig:beta-rho} for the graph of $\hat{\imath}_\beta$ and the end of the proof of Lemma~\ref{lmm:DexpOmegaPsi} for key ingredients of an analysis involving the implicit function theorem on~\eqref{eq:F-beta-t}---shows that it is a monotonically nondecreasing function of $\beta$ on its domain $\beta>0$, smooth everywhere except at $\beta_1$ and $\beta_2$, and
\begin{equation}  \label{eq:inj-beta}
    \hat{\imath}_\beta = \begin{cases} 
        \sqrt{2\beta}\,\pi & \text{if $\beta\leq\beta_1$},
        \\ t^{\mathrm{r}}_\beta\sqrt2 & \text{if $\beta_1\leq\beta\leq\beta_2$},
        \\ \pi & \text{if $\beta_2 \leq \beta$}.
    \end{cases}
\end{equation}

\subsection{Resulting bounds for the canonical and Euclidean metrics}

For the canonical and Euclidean metrics, section~\ref{sec:extreme-p} and Theorem~\ref{thm:inj} yield the following bounds on the injectivity radius of the Stiefel manifold.
\begin{corollary}  \label{cor:can-Euc}
    For the Stiefel manifold~\eqref{eq:Stiefel} with $2 \leq p \leq n-2$:
    \begin{enumerate}
        \item In the canonical metric (namely \eqref{eq:g} with $\beta = \frac12$, i.e., $\alpha = 0$), we have
\[
\mathrm{inj}_{\mathrm{St}(n,p)} \leq t^{\mathrm{r}}\sqrt2 \approx 0.91326189159122 \, \pi,    
\]
where $t^{\mathrm{r}}$ is the smallest positive solution of $\frac{\sin t}{t} + \cos t = 0$.
        \item In the Euclidean metric (namely \eqref{eq:g} with $\beta = 1$, i.e., $\alpha = -\frac12$), we have
        \[
        \mathrm{inj}_{\mathrm{St}(n,p)} \leq \pi.    
        \]
    \end{enumerate}
    For $p=n-1$, we have  $\mathrm{inj}_{\mathrm{St}(n,n-1)} = \pi$ in the canonical metric and $\mathrm{inj}_{\mathrm{St}(n,n-1)} \leq \pi$ in the Euclidean metric. For $p=1$ (with $n\geq2$), $\mathrm{inj}_{\mathrm{St}(n,1)} = \pi$. 
\end{corollary}
\begin{proof}
For the canonical case with $2 \leq p \leq n-2$, it is readily seen that the term $t^{\mathrm{r}}_\beta\sqrt2$ achieves the minimum in~\eqref{eq:inj}. For the Euclidean case with $2 \leq p \leq n-2$, one obtains $t^{\mathrm{r}}_\beta = \pi$, and the result follows from~\eqref{eq:inj}. The case $p=n-1$ with $p \geq 2$ in the Euclidean metric follows from~\eqref{eq:inj-St-n-1}. The case $p=n-1$ in the canonical metric and the case $p=1$ were mentioned in section~\ref{sec:extreme-p}.
\end{proof}

\section{Numerical investigation}
\label{sec:numerical}

This section presents numerical experiments \linebreak based on an algorithm (Algorithm~\ref{alg:semi}) that takes as input $\rho>0$ and has the following properties. 
If $\rho \leq \mathrm{inj}_{\mathrm{St}_\beta(n,p)}$, then it does not return, i.e., it loops forever.
If $\rho > \mathrm{inj}_{\mathrm{St}_\beta(n,p)}$ and $\beta\leq\frac12$ (resp.\ $\beta>\frac12$), then it is known (resp.\ suspected) to return with probability 1; however, it returns after an amount of time that is expected to grow as the dimensions increase and as $\rho$ gets close to $\mathrm{inj}_{\mathrm{St}_\beta(n,p)}$. Experiments in low dimensions support the conjecture that the upper bounds on the injectivity radius obtained in section~\ref{sec:inj-upper} are equalities.

Let $\mathcal{G}$ denote the domain of $\varphi$
and $\mathcal{H}$ the stabilizer of $I_{n\times p}$ introduced in Section~\ref{sec:Stiefel}, where the orthogonal group is replaced by the special orthogonal group without affecting the results. Hence, in the noncanonical case (section~\ref{sec:noncanonical}, $\beta\neq\frac12$), $\mathcal{G}$ is the Lie group $\mathrm{SO}(n) \times \mathrm{SO}(p)$ endowed with the (pseudo-)Riemannian metric~\eqref{eq:overline-g} and $\mathcal{H} = \mathrm{SO}(n-p) \times \mathrm{SO}(p)$; and in the canonical case (section~\ref{sec:canonical}, $\beta = \frac12$), $\mathcal{G}$ is the Lie group $\mathrm{SO}(n)$ endowed with the Riemannian metric~\eqref{eq:can-overline-g} and $\mathcal{H} = \mathrm{SO}(n-p)$.
Let $E$ denote the identity element of $\mathcal{G}$ (i.e., $E = (I_n,I_p)$ when $\mathcal{G} = \mathrm{SO}(n) \times \mathrm{SO}(p)$ and $E = I_n$ when $\mathcal{G} = \mathrm{SO}(n)$) and let ``$\cdot$'' denote the group operation. When $\mathcal{G}$ is Riemannian ($\beta \leq \frac12$), $d_{\mathcal{G}}$ denotes the distance on $\mathcal{G}$.

In line~\ref{alg:semi:G} of Algorithm~\ref{alg:semi}, we draw $G$ by drawing $H$ from a continuous distribution on $\mathcal{H}$ and letting $G \gets \mathrm{Exp}_E(\rho\overline{\xi}_E) \cdot H$.

In line~\ref{alg:semi:Xi} of Algorithm~\ref{alg:semi}, for the noncanonical case ($\beta\neq\frac12$), by a slight abuse of notation, given $G = (Q,V) \in \mathcal{G}$, we let $\log_{\mathrm{m}}(Q,V)$ stand for $\{(\Omega, \Psi) \mid \Omega \in \log_{\mathrm{m}}(Q), \Psi \in \log_{\mathrm{m}}(V)\} \subset \mathrm{T}_E\mathcal{G}$, where $\log_{\mathrm{m}}(Q) = \argmin_{\Omega \in \exp_{\mathrm{m}}^{-1}(Q)}\|\Omega\|_{\mathrm{F}}$. Note that $\log_{\mathrm{m}}(Q)$ is a singleton, called the principal matrix logarithm of $Q$, if and only if $Q$ has no negative real eigenvalue.\footnote{
    In fact, systematically returning a skew-symmetric matrix logarithm entails some subtlety. Details can be found in the publicly available code (\url{https://github.com/smataigne/InjectivityStiefel.jl}); see also \texttt{SkewLinearAlgebra.jl} (\url{https://juliapackages.com/p/skewlinearalgebra}) and~\cite{mataigne2024eigenvalue}.
    } 
Exactly negative real eigenvalues have not been observed in our experiments. Note that, when $\mathcal{G}$ is Riemannian ($\beta \leq \frac12$), we have $\overline{g}_E(\Xi,\Xi) = d_{\mathcal{G}}^2(E,G)$.

In line~\ref{alg:semi:LXi} of Algorithm~\ref{alg:semi}, observe that $[0,1] \ni t \mapsto \varphi(\Exp_E(t\Xi))$ is a curve from $I_{n\times p}$ to $U$. Its length is computed in closed form as follows. In the noncanonical case, where $\mathcal{G} = \mathrm{SO}(n) \times \mathrm{SO}(p)$, one obtains 
\begin{equation}  \label{eq:L-noncanonical}
    L_\Xi^2 = \beta \|\Omega_{11}-\Psi\|_{\mathrm{F}}^2 + \|\Omega_{21}\|_{\mathrm{F}}^2,
\end{equation}     
where $\Xi = (\Omega,\Psi)$ and $\Omega = \left[ \begin{smallmatrix} \Omega_{11} & \Omega_{12} \\ \Omega_{21} & \Omega_{22} \end{smallmatrix} \right]$ with $\Omega_{11}$ of size $p\times p$. In the canonical case, where $\mathcal{G} = \mathrm{SO}(n)$, one obtains 
\[
    L_\Xi^2 = \frac12 \|\Xi_{11}\|_{\mathrm{F}}^2 + \|\Xi_{21}\|_{\mathrm{F}}^2,
\]
where $\Xi = \left[ \begin{smallmatrix} \Xi_{11} & \Xi_{12} \\ \Xi_{21} & \Xi_{22} \end{smallmatrix} \right]$
with $\Xi_{11}$ of size $p\times p$.  

Assume that the total space $\mathcal{G}$ is Riemannian ($\alpha \geq 0$, i.e., $\beta \leq \frac12$).
Then, in view of~\cite[eq.~(5.6)]{GuiguiMiolanePennec2023}, the condition in line~\ref{alg:semi:if} of Algorithm~\ref{alg:semi} can be replaced by $d_{\mathcal{G}}(E,G) < \rho$. However, the discussion above indicates that computing $L_\Xi$ is not more expensive than computing $d_{\mathcal{G}}(E,G)$. Moreover, $L_\Xi \leq d_{\mathcal{G}}(E,G)$, hence the replacement can only increase the average time needed by Algorithm~\ref{alg:semi} to return.

Line~\ref{alg:semi:optional} offers an opportunity to design algorithm instances that return in less time on average. We have not used this opportunity in the experiments reported here. An obstacle to making efficient use of line~\ref{alg:semi:optional} is that, whenever $\xi$ is such that $\rho < t_{\mathrm{c}}(\xi)$, there is no $G$ in $\varphi^{-1}(U)$ that yields $L_\Xi < \rho$, hence the optional search is in vain.

\begin{algorithm}
    \caption{Certificate that $\rho > \mathrm{inj}_{\mathrm{St}_\beta(n,p)}$} 
    \label{alg:semi}
    \begin{algorithmic}[1]
        \Require $n > p > 0$; $\beta>0$
        \Input $\rho>0$  
        \Loop
            \State Draw $\xi$ from a continuous distribution on the unit tangent space $\mathrm{U}_{I_{n\times p}}\mathrm{St}(n,p)$;\label{alg:semi:xi} 
            \State $U \gets \Exp_{I_{n\times p}}(\rho\,\xi)$ as in~\eqref{eq:geod-Stiefel};\label{alg:semi:U} 
            \State Draw $G$ from a continuous distribution on $\varphi^{-1}(U)$;\label{alg:semi:G}
            \State Choose $\Xi\in\mathrm{T}_{E}\mathcal{G}$ in $\log_{\mathrm{m}}(G)$.\label{alg:semi:Xi}  
            \State $L_\Xi \gets$ length of $[0,1] \ni t \mapsto \varphi(\Exp_E(t\Xi))$.\label{alg:semi:LXi} 
            \If{$L_\Xi < \rho$}\label{alg:semi:if}
                \State \textsc{Return} 
            \EndIf \label{alg:semi:endif}
            \State  (Optional) Run a few steps of a descent algorithm, initialized at $G$, to search for a smaller $L_\Xi$ and \textsc{Return} if the obtained $L_\Xi < \rho$.\label{alg:semi:optional}
        \EndLoop
    \end{algorithmic}
\end{algorithm}

When Algorithm~\ref{alg:semi} returns, it is because it has found a $\xi$ in the unit tangent space $\mathrm{U}_{I_{n\times p}}\mathrm{St}_\beta(n,p)$ together with a curve $[0,1] \ni t \mapsto \varphi(\Exp_E(t\Xi))$ between $I_{n\times p}$ and $\Exp_{I_{n\times p}}(\rho\xi)$ of length $L_\Xi$ strictly smaller than $\rho$. Hence, when it returns, Algorithm~\ref{alg:semi} provides a certificate (up to floating point errors and inaccuracies in the computation of matrix logarithms) that $\rho > \mathrm{inj}_{\mathrm{St}_\beta(n,p)}$. 

Furthermore, if $\rho > \mathrm{inj}_{\mathrm{St}_\beta(n,p)}$ and the total space $\mathcal{G}$ is Riemannian ($\alpha \geq 0$, i.e., $\beta \leq \frac12$), then Algorithm~\ref{alg:semi} returns with probability 1. It is a consequence of the following proposition, letting $\overline{\mathcal{M}} = \mathcal{G}$ and recalling that $L_\Xi \leq d_{\mathcal{G}}(E,G)$. 

\begin{proposition}  \label{prp:probability-1}
Let $\overline{\mathcal{M}}$ be a complete Riemannian manifold, $\mathcal{H}$ be a compact Lie group acting smoothly on $\overline{\mathcal{M}}$, and assume that $\mathcal{M} := \overline{\mathcal{M}}/\mathcal{H}$ admits a Riemannian metric such that the canonical projection $\varphi: \overline{\mathcal{M}} \to \mathcal{M}$ is a Riemannian submersion. Let $E\in\overline{\mathcal{M}}$ and $\rho > \mathrm{inj}_{\mathcal{M}}(\varphi(E))$. Then there exists an open set $\mathcal{U}_\xi$ in $\mathrm{U}_{\varphi(E)}\mathcal{M}$ and an open set $\mathcal{U}_H$ in $\mathcal{H}$ such that, for all $\tilde\xi \in \mathcal{U}_\xi$ and all $\tilde{H} \in \mathcal{U}_H$, it holds that $d_{\overline{\mathcal{M}}}(E, \mathrm{Exp}_E(\rho \overline{\tilde\xi}_E) \tilde{H}) < \rho$. 
\end{proposition}

\begin{proof}
Since $\overline{\mathcal{M}}$ is complete, $\mathcal{M}$ is also complete~\cite{Hermann1960}. Since $\rho > \mathrm{inj}_{\mathcal{M}}(\varphi(E))$, there exists $\xi \in \mathrm{U}_{\varphi(E)}\mathcal{M}$ such that $\rho > t_{\mathrm{c}}(\xi)$. Let $\gamma_\xi: t\mapsto \mathrm{Exp}_{\varphi(E)}(t\xi)$ and $\overline\gamma_\xi: t \mapsto \mathrm{Exp}_E(t\overline{\xi}_E)$ be its horizontal lift with $\overline\gamma_\xi(0) = E$. By definition of the cut time, $\rho > t_{\mathrm{c}}(\xi)$ is equivalent to $d_{\mathcal{M}}(\varphi(E), \gamma_\xi(\rho)) < \rho$. 
By~\cite[Proposition~5.3]{GuiguiMiolanePennec2023}, $d_{\mathcal{M}}(\varphi(E), \gamma_\xi(\rho)) = \inf_{H\in\mathcal{H}} d_{\overline{\mathcal{M}}}(E,\overline\gamma_\xi(\rho)H)$. Since the distance is continuous and $\mathcal{H}$ is compact, it follows by the extreme value theorem that
there exists $H\in\mathcal{H}$ such that $d_{\overline{\mathcal{M}}}(E, \overline\gamma_\xi(\rho)H) < \rho$. Since $d_{\overline{\mathcal{M}}}$, the group action of $\mathcal{H}$, and the function $\xi \mapsto \overline\gamma_\xi(\rho)$ are continuous, 
it follows that there exists a neighborhood $\mathcal{U}_\xi$ of $\xi$ in $\mathrm{U}_{\varphi(E)}\mathcal{M}$ and a neighborhood $\mathcal{U}_H$ of $H$ in $\mathcal{H}$ such that, for all $\tilde\xi \in \mathcal{U}_\xi$ and $\tilde{H} \in \mathcal{H}$, it holds that $d_{\overline{\mathcal{M}}}(E, \overline\gamma_{\tilde\xi}(\rho) \tilde{H}) < \rho$. 
\end{proof}

Now consider the case where $\mathcal{G}$ is not Riemannian ($\alpha < 0$, i.e., $\beta > \frac12$). Then Proposition~\ref{prp:probability-1} is no longer applicable. We suspect that Algorithm~\ref{alg:semi} still returns with probability 1 when $\rho > \mathrm{inj}_{\mathrm{St}_\beta(n,p)}$, but proof attempts are thwarted by the forthcoming Proposition~\ref{prp:log-not-shortest} and Proposition~\ref{prp:not-open}. In the next result, we again slightly abuse notation by letting $\exp_{\mathrm{m}}(\Omega,\Psi)$ denote $(\exp_{\mathrm{m}}(\Omega), \exp_{\mathrm{m}}(\Psi))$.
\begin{proposition}  \label{prp:log-not-shortest}
Let 
$\beta>1$, 
\[
\xi := \frac1{\sqrt{2\beta}} \begin{bmatrix} 0 & -1 \\ 1 & 0 \\ 0 & 0 \end{bmatrix} \in \mathrm{U}_{I_{3\times2}}\mathrm{St}_\beta(3,2),
\]
$\rho := \pi\sqrt{2\beta}$, and $U := \Exp_{I_{3\times2}}(\rho\xi) = -I_{3\times2} = \varphi(I_3,-I_2)$. Then there exists $\Xi_2 \in \exp_{\mathrm{m}}^{-1}(I_3,-I_2)$---necessarily with $\Xi_2 \notin \log_{\mathrm{m}}(I_3,-I_2)$---such that $L_{\Xi_2} < \inf\{L_\Xi \mid \Xi \in \log_{\mathrm{m}}(G), G\in\varphi^{-1}(U)\} = \rho$, where $L_\Xi$ is given by~\eqref{eq:L-noncanonical}. Hence $d_{\mathrm{St}_\beta(3,2)}(I_{3\times2},U) < \inf\{L_\Xi \mid \Xi \in \log_{\mathrm{m}}(G), G\in\varphi^{-1}(U)\} = \rho$.
\end{proposition}
\begin{proof}
The fact that $\Exp_{I_{3\times2}}(\rho\xi) = -I_{3\times2}$ follows from~\eqref{eq:geod-Stiefel}. By~\eqref{eq:fiber}, we have
\begin{align*}
\varphi^{-1}(U) &= \left\{ \left( \begin{bmatrix} R_1 & 0 \\ 0 & R_2 \end{bmatrix}, -R_1 \right) \mid R_1 \in \mathrm{SO}(2), R_2 \in \mathrm{SO}(1) \right\} 
\\ &= \{ (\Omega(\theta),\Psi(\theta)) \mid \theta \in (-\pi,\pi] \},
\end{align*}
where 
\[
\Omega(\theta) := \begin{bmatrix} \cos\theta & -\sin\theta & 0 \\ \sin\theta & \cos\theta & 0 \\ 0 & 0 & 1 \end{bmatrix}
\quad \text{and} \quad
\Psi(\theta) := -\begin{bmatrix} \cos\theta & -\sin\theta \\ \sin\theta & \cos\theta \end{bmatrix}.
\]
We have 
\[
\log_{\mathrm{m}}(\Omega(\theta)) = \begin{cases}
\begin{bmatrix} 0 & -\theta & 0 \\ \theta & 0 & 0 \\ 0 & 0 & 0 \end{bmatrix} & \text{if $\theta \in (-\pi,\pi)$},
\\ \pm \begin{bmatrix} 0 & -\pi & 0 \\ \pi & 0 & 0 \\ 0 & 0 & 0 \end{bmatrix} & \text{if $\theta = \pi$},
\end{cases}
\]
and
\[
\log_{\mathrm{m}}(\Psi(\theta)) = \begin{cases}
\begin{bmatrix} 0 & -(\theta-\mathrm{sgn}(\theta)\pi)\\ \theta-\mathrm{sgn}(\theta)\pi & 0 \end{bmatrix} & \text{if $\theta \in (-\pi,\pi]\setminus\{0\}$},
\\ \pm \begin{bmatrix} 0 & -\pi \\ \pi & 0 \end{bmatrix} & \text{if $\theta = 0$}.
\end{cases}
\]
In view of~\eqref{eq:L-noncanonical}, for all $\theta \in (-\pi,\pi]$ and all $\Xi\in\log_{\mathrm{m}}(\Omega(\theta),\Psi(\theta))$, one finds $L_\Xi = \sqrt{\beta 2 \pi^2} = \pi \sqrt{2\beta}$. Hence $\inf\{L_\Xi \mid \Xi \in \log_{\mathrm{m}}(G), G\in\varphi^{-1}(U)\} = \pi \sqrt{2\beta}$. On the other hand, let \begin{equation*}
    \Xi_2 := \left(\begin{bmatrix}
    0&a\pi&-b\pi\\
    -a\pi&0&-b\pi\\
    b\pi&b\pi&0
    \end{bmatrix},\begin{bmatrix}0&-\pi\\ \pi&0\end{bmatrix}\right),
    \end{equation*}
    with $a := \frac{2\beta}{1-2\beta}$ and $b := \sqrt{2\left(1-\frac{\beta^2}{(1-2\beta)^2}\right)}$, where $b$ is a real number since $\beta>1$. It can be checked that $\Xi_2 \in \exp_{\mathrm{m}}^{-1}(I_3,-I_2)$
    and 
\[
    L_{\Xi_2} = \sqrt{\beta2(a+1)^2\pi^2 + 2b^2\pi^2}=\pi\sqrt{\frac{2\beta}{(1-2\beta)^2}+4\left(1-\frac{\beta^2}{(1-2\beta)^2}\right)}.
\]
This yields $L_{\Xi_2} < \pi\sqrt{2\beta}$ for $\beta>1$. 
\end{proof}
Proposition~\ref{prp:log-not-shortest} exhibits $\xi \in \mathrm{U}_{I_{3\times2}}\mathrm{St}_\beta(3,2)$ and $\rho$ such that $\rho > t_{\mathrm{c}}(\xi)$
yet line~\ref{alg:semi:Xi} of Algorithm~\ref{alg:semi} yields $L_\Xi = \rho$ regardless of $G$; hence lines~\ref{alg:semi:U}--\ref{alg:semi:LXi} of Algorithm~\ref{alg:semi} never reveal that $\rho > t_{\mathrm{c}}(\xi)$ for these $\xi$ and $\rho$.
However, for all $U \in \mathrm{St}_\beta(n,p)$, by the Hopf--Rinow theorem, there is a geodesic $\gamma$ in $\mathrm{St}_\beta(n,p)$ from $I_{n\times p}$ to $U$ with length $d_{\mathrm{St}_\beta(n,p)}(I_{n\times p},U)$. This geodesic admits a horizontal lift $\overline\gamma$ through $I_{n\times p}$, which is a (possibly nonminimal) geodesic of $\mathcal{G}$, and thus $d_{\mathrm{St}_\beta(n,p)}(I_{n\times p},U) 
= L_{\overline\gamma'(0)} 
= \min \{L_\Xi \mid \Xi \in \exp_{\mathrm{m}}^{-1}(G), G\in\varphi^{-1}(U)\}$ (observe the $\exp_{\mathrm{m}}^{-1}(G)$ whereas Proposition~\ref{prp:log-not-shortest} has $\log_{\mathrm{m}}(G)$), where the $\min$ is attained at $\Xi
= \overline\gamma'(0)$,
$G = \overline\gamma(1)$. A lead would thus be, in line~\ref{alg:semi:Xi} of Algorithm~\ref{alg:semi}, to choose $\Xi$ in $\exp_{\mathrm{m}}^{-1}(G)$; if $\rho > t_{\mathrm{c}}(\xi)$, then lines~\ref{alg:semi:U}--\ref{alg:semi:LXi} of Algorithm~\ref{alg:semi} would reveal it for favorably-chosen $G$ and $\Xi$ by producing $L_\Xi < \rho$. However, leaving aside that $\exp_{\mathrm{m}}^{-1}(G)$ becomes an intricate uncountable set in case of multiple eigenvalues, our attempts to show that $G$ would be favorably chosen with nonzero probability encountered the following hurdle.
\begin{proposition}  \label{prp:not-open}
For $n \geq 3$, the matrix exponential $\exp_{\mathrm{m}}: \mathcal{S}_{\mathrm{skew}}(n) \to \mathrm{SO}(n)$ is not an open map.
\end{proposition}
\begin{proof}
We prove the claim for $n=3$; the proof readily extends to higher dimensions. Let 
\[
\Omega := \begin{bmatrix}  0 & -2\pi & 0 \\ 
    2\pi & 0 & 0 \\
    0 & 0 & 0 \end{bmatrix} = \begin{bmatrix} 1 & 1 & 0 \\
        \imath & -\imath & 0 \\
        0 & 0 & 1 \end{bmatrix} 
        \mathrm{diag}(\imath 2\pi, -\imath 2\pi, 0)
        \begin{bmatrix} 1 & 1 & 0 \\
            \imath & -\imath & 0 \\
            0 & 0 & 1 \end{bmatrix}^{-1},         
\]
where the last expression is an eigendecomposition. 
Let 
\[
Q(\epsilon) := \begin{bmatrix} 1 & 0 & 0 \\
    0 & \cos\epsilon & -\sin\epsilon \\
    0 & \sin\epsilon & \cos\epsilon \end{bmatrix}
     = \begin{bmatrix} 1 & 0 & 0 \\
        0 & 1 & 1 \\
        0 & \imath & -\imath \end{bmatrix}
    \mathrm{diag}(1, e^{-\imath\epsilon}, e^{\imath\epsilon})
    \begin{bmatrix} 1 & 0 & 0 \\
        0 & 1 & 1 \\
        0 & \imath & -\imath \end{bmatrix}^{-1},
\]
where again the last expression is an eigendecomposition. Observe that $Q(\epsilon) \in \mathrm{SO}(n)$ for all $\epsilon$.
Let $\mathcal{U}_\Omega \ni \Omega$ be a relatively open subset of $\mathcal{S}_{\mathrm{skew}}(n)$, small enough that every $\tilde\Omega \in \mathcal{U}_\Omega$ retains an eigenvector within distance, say, 0.1 of $v := \begin{bmatrix} 1 & \imath & 0 \end{bmatrix}^\top$; because 
the corresponding eigenvalue $\imath 2\pi$ of $\Omega$ is simple,
such a subset exists as a consequence of~\cite[Theorem~4.11]{Ste73}. Hence, for all $\tilde\Omega \in \mathcal{U}_\Omega$, $\exp_{\mathrm{m}}(\tilde\Omega)$ has an eigenvector within distance 0.1 of $v$. Whenever $0<|\epsilon|<\pi$, $Q(\epsilon)$ has all simple eigenvalues, hence it admits only three eigendirections, which are revealed by the above eigendecomposition; it is readily seen that no eigenvector of $Q(\epsilon)$ lies within distance of 0.1 of $v$, hence $Q(\epsilon) \in \mathrm{SO}(n) \setminus \exp_{\mathrm{m}}(\mathcal{U}_\Omega)$. However, $\lim_{\epsilon\to0} Q(\epsilon) = I_3 = \exp_{\mathrm{m}}(\Omega) \in \exp_{\mathrm{m}}(\mathcal{U}_\Omega)$. Hence $\exp_{\mathrm{m}}(\mathcal{U}_\Omega)$ is not open in $\mathrm{SO}(n)$.
\end{proof}
Nevertheless, we suspect that a result akin to Proposition~\ref{prp:probability-1} holds, namely, when $\rho > \mathrm{inj}_{\mathrm{St}_{\beta}(n,p)}$, there would exist an open set $\mathcal{U}_\xi$ in $\mathrm{U}_{\varphi(E)}\mathcal{M}$ and an open set $\mathcal{U}_H$ in $\mathcal{H}$ such that, for all $\tilde\xi \in \mathcal{U}_\xi$ and all $\tilde{H} \in \mathcal{U}_H$, it would hold that $L_{\log_{\mathrm{m}}(\mathrm{Exp}_E(\rho \overline{\tilde\xi}_E) \tilde{H})} < \rho$. 
This concludes the discussion of Algorithm~\ref{alg:semi} in the case where $\mathcal{G}$ is not Riemannian ($\alpha < 0$, i.e., $\beta > \frac12$).

\begin{figure}[h]
    \centering
    \begin{minipage}{0.45\textwidth}
        \centering
        \includegraphics[width = 1\textwidth]{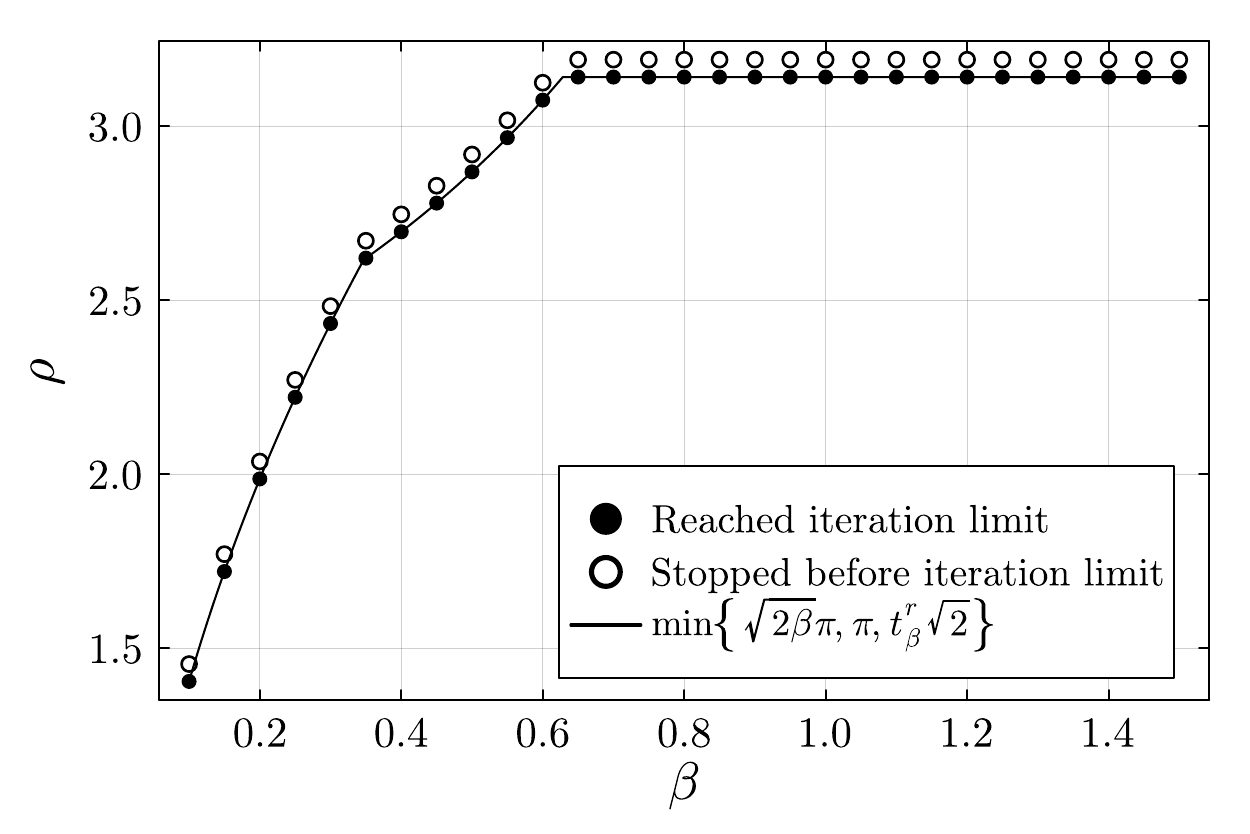} 
        \caption{Numerical experiments using Algorithm~\ref{alg:semi} for values of $\beta$ going from $0.1$ to $1.5$, spaced by $0.05$. For each value of $\beta$, we ran an experiment for $\rho = \hat{\imath}_\beta$ and $\rho = \hat{\imath}_\beta + 0.05$, where $\hat{\imath}_\beta$ is the upper bound defined in~\eqref{eq:inj}. The plot reports a black dot if Algorithm~\ref{alg:semi} reached a prescribed iteration limit (as large as possible, subject to the figure being generated within reasonable time), and a white dot if Algorithm~\ref{alg:semi} returned (hence providing a certificate that $\rho > \mathrm{inj}_{\mathrm{St}_\beta(n,p)}$). This figure was produced with $(n,p)=(4,2)$. We obtained the same figure for other values of $n$ and $p$ with $2\leq p \leq n-2$.}
        \label{fig:beta-rho}
    \end{minipage} \hfill
    \begin{minipage}{0.45\textwidth}
        \centering
        \includegraphics[width = \textwidth]{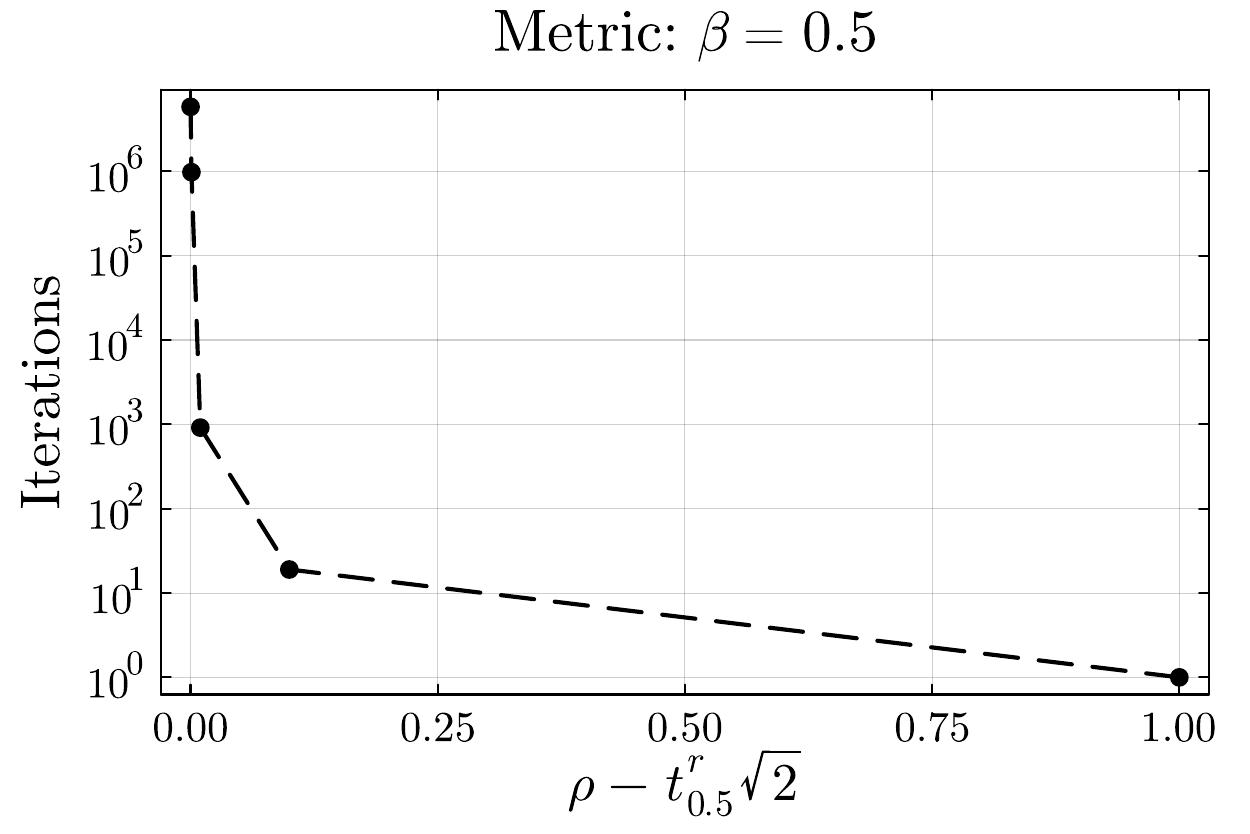}
        \caption{Numerical experiments using Algorithm~\ref{alg:semi} with $(n,p)=(4,2)$ and $\beta=0.5$ for $\rho = \hat{\imath}_\beta+\delta$ where $\delta\in\{1,0.1,0.01,0.001,0.0001\}$. The plot reports the number of iterations required by Algorithm~\ref{alg:semi} to return, as observed in representative runs. Note that, in view of the random nature of Algorithm~\ref{alg:semi}, the values differ between runs. The purpose of this figure is to show how a typical return time evolves as the upper dot located above $\beta=\frac12$ in Figure~\ref{fig:beta-rho} is brought closer to the line.
    }
    \label{fig:random_shooting_05}
    \end{minipage}
\end{figure}

Figures~\ref{fig:beta-rho} and~\ref{fig:random_shooting_05}, as well as
other figures available together with the code\footnote{\url{https://github.com/smataigne/InjectivityStiefel.jl}. Figures of the type of Figure~\ref{fig:random_shooting_05} vary between runs due to the two ``Draw'' instructions in Algorithm~\ref{alg:semi}. For the same reason, any of the white dots in figures of the type of Figure~\ref{fig:beta-rho} may turn to black, but this is all the more unlikely that the prescribed iteration limit is large. On the other hand, if any of the black dots turns to white, then either Conjecture~\ref{conj:inj} is disproved, or the reason has to be found in the inaccuracies in the execution of the steps of Algorithm~\ref{alg:semi}.} that produced them, support the following conjecture. 
\begin{conjecture}  \label{conj:inj}
\eqref{eq:inj-St-n-1} and~\eqref{eq:inj} hold with an equality.
\end{conjecture}
As mentioned in Section~\ref{sec:introduction}, the conjecture was proved in the recent eprint~\cite{zimmermann2024injectivity} for $\beta = 1$.

Although the experiments strikingly corroborate Conjecture~\ref{conj:inj}, they do not qualify as a proof for several reasons: only a few pairs $(n,p)$ are considered; only finitely many points in the $(\beta,\rho)$ space are tested; the time budget may have been chosen too small for some of the sampled $(\beta,\rho)$'s, leading to the incorrect belief that Algorithm~\ref{alg:semi} does not return for those $(\beta,\rho)$'s; the accuracy of several operations in Algorithm~\ref{alg:semi} depends on the floating point accuracy and on the tolerance set in the algorithms that compute the matrix exponential and logarithm. Moreover, when $\beta>\frac12$, there is no proof that Algorithm~\ref{alg:semi} returns with probability 1 whenever $\rho > \mathrm{inj}_{\mathrm{St}_\beta(n,p)}$; hence, when $\beta>\frac12$, even if we became confident that Algorithm~\ref{alg:semi} does not return for some $\rho$, this would not imply that $\mathrm{inj}_{\mathrm{St}_\beta(n,p)} \geq \rho$.

Nevertheless, it is tempting to bet that the upper bounds~\eqref{eq:inj-St-n-1} and~\eqref{eq:inj} will never be improved.

\section{Concluding remarks}
\label{sec:conclusion}

The first few days of March 2024 have seen an intense activity around the Stiefel manifold. First, the eprint~\cite{zimmermann2024high} was submitted, 
proving that, whenever $2 \leq p \leq n-2$, the supremum of the sectional curvature of $\mathrm{St}(n,p)$ with the canonical metric (i.e.,~\eqref{eq:g} with $\beta=\frac12$) is $5/4$. As shown in~\cite[v2, Corollary~10]{zimmermann2024high}, it follows that $\mathrm{inj}_{\mathrm{St}_{\beta = \frac12}(n,p)} \geq \sqrt{\frac45} \pi \approx 0.89442719099991\,\pi$. 

A few hours later, the first eprint (see \href{https://arxiv.org/abs/2403.02079}{arXiv:2403.02079}) of the present work was submitted. Combining~\cite[v2, Corollary~10]{zimmermann2024high} with Corollary~\ref{cor:can-Euc} yields the following bounds on the injectivity radius of the Stiefel manifold with the canonical metric: 
\[
    0.894\,\pi 
    \ < \ \sqrt{\frac45}\,\pi \ \leq \  \mathrm{inj}_{\mathrm{St}_{\beta = \frac12}(n,p)} \ \leq \ t^{\mathrm{r}}_{\beta = \frac12} \sqrt2 \ < \ 
    0.914\,\pi
\] 
for $2 \leq p \leq n-2$, where $t^{\mathrm{r}}_{\beta = \frac12}$ is the first positive solution of $\frac{\sin t}{t} + \cos t = 0$. Moreover, according to Conjecture~\ref{conj:inj}, the injectivity radius is believed to be equal to the endpoint $t^{\mathrm{r}}_{\beta = \frac12} \sqrt2$ of the interval.

Two days later, the eprint~\cite{stoye2024injectivity} was submitted, giving explicit expressions for all Jacobi fields along a specific geodesic of $\mathrm{St}_{\beta = \frac12}(n=4,p=2)$.
These expressions indicate that, when $(n,p) = (4,2)$ and $\beta = \frac12$, the conjugate points given by Theorem~\ref{thm:conjugate} are first conjugate points (which, in view of Theorem~\ref{thm:ab}, has to be the case if Conjecture~\ref{conj:inj} holds), and that their multiplicity is $1$.

One can expect that more is to come, as the Riemannian geometries of the Stiefel manifold still hold many secrets. 

\bibliographystyle{siamplain}
\bibliography{../../bib/pabib}

\end{document}